\documentclass[11pt]{amsart} 
\usepackage[english]{babel}
\usepackage{graphicx,epsfig}
\usepackage{bbm}
\usepackage{cite}
\usepackage{amsthm}
\usepackage{esvect}
\usepackage{amsmath,amssymb,latexsym, amsfonts, amscd, amsthm, xy}
\input{xy}
\xyoption{all}

\usepackage{mathrsfs}

\usepackage{mathtools}

\usepackage[margin=1in]{geometry}

\usepackage[dvipsnames]{xcolor}

\makeindex 

\usepackage{hyperref}
\hypersetup{pdftoolbar=true, pdftitle={GKform}, pdffitwindow=true, colorlinks=true, citecolor=blue, filecolor=black, linkcolor=purple, urlcolor=red, hypertexnames=false}
\theoremstyle{plain}
\newtheorem{theorem}{Theorem}[section]

\newtheorem{proposition}[theorem]{Proposition}
\newtheorem{corollary}[theorem]{Corollary}

\newtheorem*{assumption*}{Assumption}
\newtheorem{lemma}[theorem]{Lemma}

\theoremstyle{definition}
\newtheorem{definition}[theorem]{Definition}

\newtheorem*{goal*}{Goal}
\newtheorem*{problem*}{Comment}

\newtheorem{notation}[theorem]{Notation}

\theoremstyle{remark}
\newtheorem{remark}[theorem]{Remark}

\def\lra{{\longrightarrow}}
\def\SL{{\rm SL}}

\def\PP{{\mathbb P}}

\DeclareMathOperator{\Gr}{Gr}

\DeclareMathOperator{\tors}{tors}

\DeclareMathOperator{\quot}{quot}

\DeclareMathOperator{\comp}{comp}

\DeclareMathOperator{\Fil}{Fil}

\DeclareMathOperator{\sym}{Sym}

\DeclareMathOperator{\dR}{dR} \DeclareMathOperator{\pr}{pr}
\DeclareMathOperator{\id}{id}

 \DeclareMathOperator{\cl}{cl}

\DeclareMathOperator{\et}{et}

\DeclareMathOperator{\End}{End} \DeclareMathOperator{\CH}{CH}
\DeclareMathOperator{\AJ}{AJ} \DeclareMathOperator{\Aut}{Aut}

\DeclareMathOperator{\spec}{Spec}

 \DeclareMathOperator{\Frob}{Frob}

\DeclareMathOperator{\Isog}{Isog}

\DeclareMathOperator{\alg}{alg}

\DeclareMathOperator{\HC}{HC}
\DeclareMathOperator{\GHC}{GHC}

\DeclareMathOperator{\Jac}{Jac}

\def\cN{\mathcal{N}}

\def\d{\mathrm{d}}

\def\Z{\mathbb{Z}}

\def\F{\mathbb{F}}

\def\Q{\mathbb{Q}}

\def\C{\mathbb{C}}

\def\R{\mathbb{R}}

\def\bdf{\begin{defn}}
\def\edf{\end{defn}}

\def\cH{\mathcal{H}}

\def\cO{\mathcal{O}}

\def\cE{\mathcal{E}}

\def\Gal{{\rm Gal}}

\def\corr{{\rm Corr}}

\def\ab{{\rm ab}}

\def\ab{\text{ab}}

\def\d1{d^{(1)}}

\def\d{\mathbf{d}}

\def\oh{\mathcal{O}}

\usepackage{tikz}
\usetikzlibrary{positioning} 
\usepackage{tikz-cd}
\usetikzlibrary{arrows,graphs,decorations.pathmorphing,decorations.markings}

\tikzset{
commutative diagrams/.cd,
arrow style=tikz,
diagrams={>=latex}}

\usepackage{microtype}
\begin{document} 

 \title[Heegner cycles in Griffiths groups of Kuga--Sato varieties]{\small Heegner cycles in Griffiths groups of Kuga--Sato varieties}
\author{David T.-B. G. Lilienfeldt}
%\thanks{DL partially supported by an Alexis and Charles Pelletier Fellowship.}
 \address{Department of Mathematics and Statistics, McGill University, Montreal, Canada}
\email{david.lilienfeldt@mail.mcgill.ca}
\curraddr{Einstein Institute of Mathematics, Hebrew University of Jerusalem, Israel}
\email{davidterborchgram.lilienfeldt@mail.huji.ac.il}
\date{\today}

\begin{abstract}
The aim of this article is to prove, using complex Abel--Jacobi maps, that the subgroup generated by Heegner cycles associated with a fixed imaginary quadratic field in the Griffiths group of a Kuga--Sato variety over a modular curve has infinite rank. This generalises a classical result of Chad Schoen for the Kuga--Sato threefold, and complements work of Amnon Besser on complex multiplication cycles over Shimura curves. The proof relies on a formula for the images of Heegner cycles under the complex Abel--Jacobi map given in terms of explicit line integrals of even weight cusp forms on the complex upper half-plane. The latter is deduced from previous joint work of the author with Massimo Bertolini, Henri Darmon, and Kartik Prasanna by exploiting connections with generalised Heegner cycles. As a corollary, it is proved that the Griffiths group of the product of a Kuga--Sato variety with powers of an elliptic curve with complex multiplication has infinite rank. This recovers results of Ashay Burungale by a different and more direct approach.
\end{abstract}

\maketitle

\tableofcontents

\section{Introduction}

\subsection{Heegner cycles}

A generalisation of the conjecture of Birch and Swinnerton-Dyer \cite{BSD1, BSD2} involving algebraic cycles exists for higher dimensional algebraic varieties over number fields. It is due independently to Beilinson \cite{beilinson} and Bloch \cite{bloch}. The motive of a newform $f$ of level $\Gamma_1(N)$ and higher even weight $k+2$, with $k=2r\geq 2$, is cut out from the Kuga--Sato variety $W_k$ over $\Q$ of dimension $k+1$ and level $\Gamma_1(N)$ by the work of Scholl \cite{scholl}. The conjecture of Beilinson and Bloch roughly predicts that the order of vanishing of the $L$-function of $f$ over a number field $F$ at its center $s=r+1$ is accounted for by the existence of non-torsion elements in the Chow group $\CH^{r+1}(W_{k,F})_0$ of null-homologous algebraic cycles of codimension $r+1$ modulo rational equivalence. 

Given an imaginary quadratic field $K$ satisfying the Heegner hypothesis with respect to $N$ (all primes dividing $N$ split in $K$), a construction of cycles that could potentially account for the first central derivative of $L(f/K, s)$ was envisioned in the seminal work of Gross and Zagier \cite[\S V. 4]{GZ}. These are higher dimensional analogues of Heegner points known as Heegner cycles. They live in complex multiplication (CM) fibres of the Kuga--Sato variety $W_k \lra X_1(N)$, and lie above Heegner points via the map of modular curves $X_1(N)\lra X_0(N)$. Zhang \cite{shouwu} has proved a Gross--Zagier type formula relating $L'(f/K, r+1)$ to the Beilinson--Bloch height of a Heegner cycle. This formula has recently been generalised by Qiu \cite{congling}. A $p$-adic version of Zhang's formula has been obtained by Nekov\'{a}\v{r} \cite{nekovar}, with a step in the proof filled by Shnidman \cite{shnidman}, generalising previous work of Perrin-Riou \cite{perrinriou} for weight $2$ forms. A universal $p$-adic Gross--Zagier formula encompassing the previously known formulae has recently been obtained by Disegni \cite{disegniuniversal}. 
%The latter should allow the adaptation of the method of \cite{burungaledisegni} for proving non-vanishing results for $p$-adic height pairings of CM abelian varieties to cover the case of Heegner cycles, as remarked therein.
 Kolyvagin's \cite{koly88, grosskoly} method of Euler systems has also been adapted to the setting of Heegner cycles by Nekov\'{a}\v{r} \cite{NekovarKS}. 

The present article is concerned with questions about the algebraic geometric, or Hodge theoretic, incarnation of Heegner cycles. The aim is to give an explicit formula for their images under the complex Abel--Jacobi map, and deduce consequences for Griffiths groups of Kuga--Sato varieties. This in turn implies results about certain variants of generalised Heegner cycles introduced by Bertolini, Darmon, and Prasanna in \cite{bdp3}. 

\subsection{The Griffiths group}\label{s:intro_gr}

The Griffiths group $\Gr(X)=\bigoplus_{j=0}^{d} \Gr^j(X)$  of a smooth projective algebraic variety $X$ of dimension $d$ is the group of null-homologous algebraic cycles modulo algebraic equivalence. It is a rather mysterious quotient of the null-homologous Chow group. For codimension $1$ algebraic cycles, algebraic equivalence coincides with homological equivalence, hence $\Gr^1(X)$ is trivial. In higher codimension, this is however no longer the case. 
Indeed, Griffiths \cite{griffiths} showed that for a general quintic hypersurface $X$ of $\PP^4$ over $\C$, $\Gr^2(X)\otimes_{\Z} \Q$ is non-zero. Clemens \cite{clemens} then showed in the same case that $\dim_\Q \Gr^2(X)\otimes_{\Z} \Q=\infty$. Ceresa \cite{ceresa} proved that the Ceresa cycle $\iota(C)-[-1]^* \iota(C)\in \Gr^{g-1}(\Jac(C))$ is non-torsion for a generic curve of genus $g\geq 3$ over $\C$, where $\iota: C \hookrightarrow \Jac(C)$ is a fixed Abel--Jacobi embedding, and Nori \cite{Nori} later proved in the same case that $\dim_{\Q} \Gr^2(\Jac(C))\otimes_{\Z} \Q=\infty$. In recent related developments, Totaro \cite{totaro} showed that for a very general principally polarised complex abelian $3$-fold $X$, $\Gr^2(X)\otimes \Z/\ell\Z$ is infinite for any prime $\ell$ (see \cite{schoenmodulo, rosenschonsrinivas} for prior results).
%All of these results depend on the presence of parameters in the base field. 

Over number fields, the first explicit example of a variety for which such phenoma occurred was found by Harris \cite{harris} who studied the Ceresa cycle of the Fermat quartic over $\Q$, and proved that it is non-zero modulo algebraic equivalence by computing its image under the complex Abel--Jacobi map. Bloch \cite{bloch} then proved that the algebraic equivalence class of the Ceresa cycle of the Fermat quartic is non-torsion using a purely algebraic method involving the \'etale Abel--Jacobi map. 
%This was for him the first piece of evidence for his ``recurring fantasy'', which is known today as the Beilinson--Bloch conjecture. 
Schoen \cite{schoen} studied Heegner cycles on the Kuga--Sato threefold $W$ of level $\Gamma(N)$ and proved that $\dim_{\Q} \Gr^2(W_{\bar{\Q}})\otimes_{\Z} \Q=\infty$. His method cleverly combines a complex Abel--Jacobi calculation with the algebraic method pioneered by Bloch. The approach of Schoen was generalised by Besser \cite{besser95} to Heegner cycles over Shimura curves associated with indefinite division quaternion algebras over $\Q$. 

\subsection{Generalised Heegner cycles}

Let $K$ be an imaginary quadratic field satisfying the Heegner hypothesis with respect to $N$. Let $A$ be an elliptic curve with CM by the maximal order $\oh_K$ of $K$ over the Hilbert class field $H$ of $K$. Fix an embedding $H\hookrightarrow \C$ such that $A_\C = \C/\oh_K$.
Bertolini, Darmon, and Prasanna introduced in \cite{bdp1} a distinguished collection of cycles, known as generalised Heegner cycles, on the product varieties $W_{r,H}\times_H A^r$ where $r\geq 1$ is an integer. These cycles account for the motives of cusp forms twisted by certain algebraic Hecke characters of infinite order. In a subsequent paper \cite{bdp3}, they further introduced variants of generalised Heegner cycles on the product varieties $W_{r_1,H}\times_H A^{r_2}$ where $r_1\geq r_2$ are non-negative integers of the same parity, and proved some non-vanishing results in Griffiths groups using $p$-adic Hodge theoretic methods. 
%Suppose that $A$ is an elliptic curve over $\Q$ with CM by the maximal order of an imaginary quadratic field $K$ of class number $1$, and $0\leq i \leq r-1$ with $i\equiv r-1 \pmod 2$.  Then, using a variant of their construction, they produced in \cite{bdp3} explicit non-torsion elements in the Griffiths group $\Gr^{r+1}((W_{r+i}\times A^{r-i})_K)$ contingent on certain vanishing conditions for the relevant $L$-functions. Their method is purely $p$-adic Hodge theoretic and uses the $p$-adic syntomic Abel--Jacobi map.
Burungale \cite{burungale17, burungale20} studied the $p$-adic syntomic Abel--Jacobi images of these variants of generalised Heegner cycles modulo $p$ (both in the case of modular curves and Shimura curves). A consequence of his work is that the subgroups generated by these cycles in the relevant Griffiths groups have infinite rank. His method uses tools from Iwasawa theory and crucially relies on the $p$-adic Gross--Zagier formula for generalised Heegner cycles of \cite{bdp1} and \cite{brooks}.
In joint work of the author with Bertolini, Darmon, and Prasanna \cite{bdlp}, the original approach of Schoen \cite{schoen} was adapted to the generalised Heegner cycles of \cite{bdp1}. In the case when $r\geq 2$ and the discriminant of $K$ is not $-3$ or $-4$ (assumed for simplicity), the main result of \cite{bdlp} proves that $\dim_{\Q} \Gr^{r+1}((W_{r,H}\times_H A^r)_{\bar \Q})\otimes_{\Z} \Q=\infty$ by exploiting a complex Abel--Jacobi calculation for generalised Heegner cycles. It is worth pointing out that the proof actually gives the infinitude of the rank of the Griffiths group over the maximal abelian extension $K^{\ab}$ of $K$.
%In particular, a formula for the image of generalised Heegner cycles under the complex Abel--Jacobi map was obtained.

\subsection{Main results}

Let $N\geq 5$ and $r\geq 1$ be integers, and let $k:=2r$. Let $W_k$ denote the Kuga--Sato variety over $\Q$ fibred over the modular curve $X_1(N)$ (defined in Section \ref{s:KS}).
In \cite[Remark 10]{bdlp}, it was noted that the techniques developed in \cite{bdlp} for generalised Heegner cycles should adapt to the case of Heegner cycles. Namely, it should be possible to establish a formula for the complex Abel--Jacobi images of Heegner cycles, and use such a formula to deduce that $\dim_{\Q} \Gr^{r+1}(W_{k, \bar{\Q}})\otimes_{\Z} \Q=\infty$. The first goal of the present article is to carry out this program.
% and fill this gap in the literature. 
This complements the work of Besser \cite{besser95}, which does not treat the case of the quaternion algebra $M_2(\Q)$.  The second goal is to use the result for $W_k$ to deduce similar results for products of $W_k$ with even powers of CM elliptic curves.

\subsubsection{The complex Abel--Jacobi map}\label{ss:CAJ}
The complex Abel--Jacobi map
\[
\AJ_{W_k}^{r+1} : \CH^{r+1}(W_{k,\C})_0 \lra J^{r+1}(W_{k,\C}):=\frac{(\Fil^{r+1} H^{k+1}_{\dR}(W_{k,\C}))^\vee}{H_{k+1}(W_{k,\C}(\C), \Z)}, 
\]
is a homomorphism from the codimension $r+1$ null-homologous Chow group to the Griffiths intermediate Jacobian of $W_{k,\C}$, a complex torus. The complex vector space $S_{k+2}(\Gamma_1(N))$ of holomorphic cusp forms of weight $k+2$ and level $\Gamma_1(N)$ is naturally identified with $H^{k+1,0}(W_{k,\C})$ via the association $f \mapsto \omega_f$ where $\omega_f(\C/\langle 1,\tau\rangle, 1/N):=f(\tau)(2\pi i dw)^{k}\otimes (2\pi i d\tau)$ for $\tau$ in the complex upper half-plane $\cH$ and $w$ the standard coordinate on the torus $\C/\langle 1,\tau\rangle$ with lattice $\langle 1,\tau\rangle:=\Z \oplus \Z\tau$.

\subsubsection{Setup and assumptions}\label{ss:setup}

Let $N\geq 5$ and $r\geq 1$ be integers, and let $k:=2r$.
Let $K$ be an imaginary quadratic field satisfying the Heegner hypothesis with respect to $N$: all primes dividing $N$ are split in $K$. We impose no restrictions on the discriminant $-d_K$ of $K$ (whereas \cite{bdlp} assumed $-d_K\neq -3, -4$ for simplicity). Choose an ideal $\cN$ of $\oh_K$ such that $\oh_K/\cN = \Z/N\Z$ (which exists by the Heegner hypothesis). Let $A$ be an elliptic curve with CM by the maximal order $\oh_K$ of $K$ over the Hilbert class field $H$ of $K$. Fix an embedding $H\hookrightarrow \C$ such that  $A_\C = \C/\oh_K$ and a choice of $\Gamma_1(N)$-level structure $t\in A[\cN]$ (i.e., a generator $t$ of the cyclic group $A[\cN]$). Associated to any (isomorphism class of) isogeny $\varphi : A \lra A'$ of elliptic curves whose kernel intersects $A[\cN]$ trivially is a Heegner cycle denoted $\tilde\Delta_\varphi^{\HC} \in \CH^{r+1}(W_{k, \bar \Q})_0$ (defined in Section \ref{s:HC}). 

\subsubsection{The complex Abel--Jacobi formula}
By definition, $\epsilon_{W_k} \tilde\Delta_\varphi^{\HC}=\tilde\Delta_\varphi^{\HC}$ where $\epsilon_{W_k}$ denotes Scholl's projector with rational coefficients (defined in Section \ref{s:KS}). Let $\tilde\epsilon_{W_k}$ denote the normalised correspondence with integral coefficients (see Definition \ref{rem:intcycles}). By functoriality of the complex Abel--Jacobi map, we will solely be interested in the piece of the Abel--Jacobi map that survives after composing with $\tilde\epsilon_{W_k}$: 
\[
\AJ_{W_k}:=\tilde\epsilon_{W_k}\circ \AJ_{W_k}^{r+1} : \CH^{r+1}(W_{k,\C})_0 \lra \frac{(\Fil^{r+1} \tilde\epsilon_{W_K} H^{k+1}_{\dR}(W_{k,\C}))^\vee}{\Pi_k},
\]
with $\Pi_k:=\tilde\epsilon_{W_k} H_{k+1}(W_{k,\C}(\C), \Z)$. By properties of Scholl's projector $\epsilon_{W_k}$, the complex vector space $\Fil^{r+1} \epsilon_{W_K} H^{k+1}_{\dR}(W_{k,\C})$ is identified with $S_{k+2}(\Gamma_1(N))$ via the association described in Section \ref{ss:CAJ} (see Proposition \ref{prop:epsiloncusp}). We may thus view $\AJ_{W_k}(\tilde\Delta_\varphi^{\HC})$ as an element of $S_{k+2}(\Gamma_1(N))^\vee$ modulo some lattice. In Section \ref{s:caj}, we will define a slightly larger lattice $L'_k$ in the dual space of cusp forms of weight $k+2$ and level $\Gamma_1(N)$, which has the advantage that it allows for more explicit formulae. 
The first main result is a formula for the complex Abel--Jacobi image of a Heegner cycle viewed in the torus $S_{k+1}(\Gamma_1(N))^\vee/L'_k$:

\begin{theorem}\label{intro:thm:CAJ}
With the assumptions of Section \ref{ss:setup},
let $\varphi : \C/\oh_K \lra \C/\langle 1, \tau' \rangle$ be an isogeny of degree $d_\varphi$ whose kernel intersects $A[\cN]$ trivially and such that $\varphi(t)=\frac{1}{N} \pmod{\langle 1, \tau' \rangle}$. For any cusp form $f$ of weight $k+2$ and level $\Gamma_1(N)$, we have the following equality modulo the lattice $L'_k$:
 \[
 \AJ_{W_k}((2^k k!)^2 d_\varphi^r\tilde\Delta_{\varphi}^{\HC})(\omega_f)=\frac{(-2\sqrt{-d_K})^rd_\varphi^k(2\pi i)^{r+1}(2^{2k}N^k(k!)^2)^2}{(\tau'-\bar{\tau}')^r} \int_{i\infty}^{\tau'}(z-\tau')^r(z-\bar{\tau}')^r f(z) dz.
 \]
 \end{theorem}
 
%Here $\omega_f$ is the de Rham cohomology class of $W_k$ attached to $f$, $\Delta_{\varphi}^{\HC}$ is the Heegner cycle attached to the isogeny $\varphi$, and $L'_k$ is a lattice in the dual of the space of cusp forms of weight $k+2$ and level $\Gamma_1(N)$ defined in Section \ref{s:caj}. 

%\footnote{
%After writing this paper, it came to the author's attention that in her thesis \cite{hopkinsthesis}, Hopkins had computed the complex Abel--Jacobi image of Heegner cycles by adapting the proof of \cite[Theorem 1]{bdlp} verbatim. Although partially announced in \cite{hopkins}, these results were never published. 
%}

\begin{remark}
The proof of Theorem \ref{intro:thm:CAJ} that we give avoids any adaptation of the Abel--Jacobi calculations of \cite{bdlp} (although the method of \cite{bdlp} can be adapted). 
Instead, by exhibiting correspondences from $W_{k,H}\times_H A^k$ to $W_{k,H}$ which map generalised Heegner cycles to Heegner cycles (see Proposition \ref{prop:relation}), we use functorial properties of the complex Abel--Jacobi map to deduce Theorem \ref{intro:thm:CAJ} directly from the formula for generalised Heegner cycles \cite[Theorem 1]{bdlp}.  
\end{remark}

\begin{remark}
Theorem \ref{intro:thm:CAJ}  implies the compatibility of the conjectural partial generalisations of the Gross--Kohnen--Zagier theorem to higher weights formulated by Hopkins in \cite{hopkins} with the conjectures of Beilinson and Bloch. This is discussed further in Remark \ref{rem:hopkins}.
\end{remark}

\subsubsection{Griffiths groups of Kuga--Sato varieties}

Using Theorem \ref{intro:thm:CAJ}, we prove the second main theorem:
\begin{theorem}\label{intro:main:thm}
With the assumptions of Section \ref{ss:setup}, the subgroup of $\Gr^{r+1}(W_{k, K^{\ab}})$ generated by the algebraic equivalence classes of the Heegner cycles $\tilde\Delta_{\varphi}^{\HC}$ indexed by isomorphism classes of isogenies $\varphi : A\lra A'$ whose kernels intersect $A[\cN]$ trivially has infinite rank.
\end{theorem}

\begin{remark}
Theorem \ref{intro:main:thm} implies in particular that $\dim_{\Q} \Gr^{r+1}(W_{k, \bar \Q})\otimes_{\Z} \Q=\infty$. In the case when $W_k$ is a threefold (the case $r=1$) and the congruence subgroup is $\Gamma(N)$, the latter is Schoen's main theorem in \cite{schoen}. However, Schoen's proof proceeds by studying Heegner cycles attached to varying imaginary quadratic fields, so even in the case $r=1$, Theorem \ref{intro:main:thm} is a strengthening of his result. As already noted, Theorems \ref{intro:thm:CAJ} and \ref{intro:main:thm} together with their proofs complement the work of Besser \cite{besser95}, which is valid for Kuga--Sato varieties over indefinite quaternionic Shimura curves.
\end{remark}

\begin{remark}
Let $F$ be a number field and fix a prime $\ell$. The $\ell$-adic \'etale Abel--Jacobi map \cite{bloch}
\[
\AJ^{r+1}_{\et} : \CH^{r+1}(W_{k,F})_0 \lra H^1(\Gal(\bar \Q / F), H^{k+1}_{\et}(W_{k,\bar \Q}, \Q_\ell(r+1))),
\]
is a homomorphism from the codimension $r+1$ null-homologous Chow group of cycles rational over $F$ to the first (continuous) Galois cohomology group of the $\Gal(\bar\Q/F)$-module $H^{k+1}_{\et}(W_{k,\bar \Q}, \Q_\ell(r+1))$. It is conjectured, for cycles defined over number fields, to be injective up to torsion \cite[Conjecture 9.15]{jannsen}. In the course of proving Theorem \ref{intro:main:thm}, we (roughly) show using Theorem \ref{intro:thm:CAJ} that $\Delta_{\varphi}^{\HC}$ has infinite order in the Griffiths group asymptotically as $\deg(\varphi)$ goes to infinity. Thus, conjecturally, our results imply asymptotic non-vanishing results for $\ell$-adic \'etale Abel--Jacobi images of Heegner cycles.
\end{remark}

\subsubsection{Griffiths groups of products of Kuga--Sato varieties with even powers of CM elliptic curves}

We next turn our attention to the variants of generalised Heegner cycles introduced in \cite{bdp3}. Retain the notations and assumptions of Section \ref{ss:setup}. Given an even integer $0\leq k'=2r'\leq k$, the variants of generalised Heegner cycles are cycles $\tilde\Delta_{k,k', \varphi}$ of codimension $r+r'+1$ on $(W_{k, H}\times_H A^{k'})_{K^{\ab}}$ indexed by isomorphism classes of isogenies $\varphi : A \lra A'$ whose kernels intersect $A[\cN]$ trivially (defined in Section \ref{s:HC}). In the case $k'=0$ these are Heegner cycles, while in the case $k'=k$ they are the generalised Heegner cycles of \cite{bdp1}. We prove that Heegner cycles are images under certain correspondences of these variants of generalised Heegner cycles (see Proposition \ref{prop:relation}). The third main result then follows from Theorem \ref{intro:main:thm}:

\begin{theorem}\label{main}
Under the assumptions of Section \ref{ss:setup}, if $0\leq k'=2r'\leq k$ is another even integer, then the subgroup of $\Gr^{r+r'+1}((W_{k,H}\times_H A^{k'})_{K^{\ab}})$ generated by the algebraic equivalence classes of variants of generalised Heegner cycles $\tilde\Delta_{k,k',\varphi}$ indexed by isomorphism classes of isogenies $\varphi : A \lra A'$ whose kernels intersect $A[\cN]$ trivially has infinite rank.
\end{theorem}

\begin{remark}
Theorem \ref{main} is a generalisation of the main result of \cite{bdlp}, which is valid under the same hypothesis assuming $k'=k$ (and $-d_K\neq -3,-4$), but without requiring $k\geq 2$ to be even. Theorems \ref{intro:main:thm} and \ref{main} recover results of Burungale \cite{burungale20} by a fundamentally different approach. The complex geometric method presented here is more direct, as it does not rely on any type of ($p$-adic) Gross--Zagier formula, which was instrumental in \cite{burungale20}.  
\end{remark}

\subsection{Strategy}

%The proof of Theorem \ref{intro:main:thm} follows the same strategy as the proof of \cite[Theorem 2]{bdlp}, which we briefly outline. 
The method of proof of Theorem \ref{intro:main:thm} follows closely that of the proof of \cite[Theorem 2]{bdlp}, which itself is an adaptation of the original work and ideas of Schoen \cite{schoen}. We give a self-contained proof which does not assume familiarity with these prior works, offering along the way some additional details, simplifications, and minor fixes.
The method can be summarised as follows.
Analytic estimates of the integrals appearing in Theorem \ref{intro:thm:CAJ} imply that infinitely many Heegner cycles have either infinite or large order in the Griffiths group. A comparison argument with Bloch's \'etale variant of the Abel--Jacobi map defined on torsion cycles, together with fundamental properties of \'etale cohomology, allows us to deduce that in fact infinitely many Heegner cycles have infinite order in the Griffiths group. Finally, using knowledge from the theory of complex multiplication about the Galois action on these cycles enables us to prove that they generate a subgroup of infinite rank in the Griffiths group. 

%In the proof of Theorem \ref{intro:main:thm} in Section \ref{s:gr}, we will content ourselves with pointing out the necessary changes to the proof of \cite[Theorem 2]{bdlp}.

\subsection{Outline}

In Section \ref{s:forms}, we recall the definition of Kuga--Sato varieties and define Scholl's projector, which cuts out spaces of cusp forms in the de Rham cohomology of these varieties. We define the product varieties $W_{k,H}\times_H A^{k'}$ for even integers $k\geq k'$, as well as correspondences relevant for the definition of the various cycles. In Section \ref{s:cycles}, we define the variants of generalised Heegner cycles following \cite{bdp3}. We exhibit certain correspondences from $W_{k,H}\times_H A^{k'}$ to $W_{k,H}$, and prove that they map variants of generalised Heegner cycles to rational multiples of Heegner cycles. In Section \ref{s:caj}, we recall the definition of the complex Abel--Jacobi map and prove Theorem \ref{intro:thm:CAJ} using the correspondences defined in the previous section with $k=k'$. The next sections are dedicated to proving Theorem \ref{intro:main:thm}. Section \ref{s:bloch} recalls basic properties of Bloch's \'etale variant of the Abel--Jacobi map on torsion cycles. Section \ref{s:finiteness} is devoted to proving a finiteness result for the \'etale cohomology groups of $W_{k,\bar{\Q}}$ with torsion coefficients related to the target of Bloch's map. In Section \ref{s:explicit}, we write down a collection of explicit isogenies that gives rise to a distinguished subcollection of Heegner cycles on which we will focus for the proof of Theorem \ref{intro:main:thm}. In Section \ref{s:asymptotics}, we use Theorem \ref{intro:thm:CAJ} to derive asymptotic information about the behaviour of Heegner cycles in our subcollection. Finally, in Section \ref{s:gr} we prove Theorem \ref{intro:main:thm} and deduce Theorem \ref{main}. 

\subsection{Notations and conventions}\label{s:convention}
All number fields in this article are viewed as embedded in a fixed algebraic closure $\bar{\Q}$ of $\Q$. Moreover, we fix a complex embedding $\bar{\Q}\hookrightarrow \C$, as well as a $p$-adic embedding $\bar{\Q}\hookrightarrow \C_p$ for each rational prime $p$. In this way, all finite extensions of $\Q$ are viewed simultaneously as subfields of $\C$ and $\C_p$. Throughout, the subscript $\Q$ on a group will denote the tensor product with $\Q$ over $\Z$. If $F$ is a field and $
X$ is a variety over a field contained in $F$, then $X_F$ will denote its base change. 
%When computing Abel--Jacobi images of cycles defined in Chow groups with rational coefficients, it will be tacitly assumed that these have been multiplied by a suitable integer in order to clear denominators. 
Given two varieties $X$ and $Y$ over a field $F$, we write $\corr^{r}(X, Y):=\CH^{\dim X+r}(X\times_F Y)$ for the group of correspondences of degree $r$.

\section{Cusp forms and Kuga--Sato varieties}\label{s:forms}

\subsection{Kuga--Sato varieties}\label{s:KS}

Let $k\geq 2$ and $N\geq 5$ be integers. Throughout, we will suppose that $k$ is even, and we let $k=2r$ with $r\geq 1$. The open modular curve $Y_1(N)$ over $\Q$ is the fine moduli space representing isomorphism classes of pairs consisting of an elliptic curve over a $\Q$-scheme together with a point of exact order $N$. It is a geometrically connected smooth affine curve over $\Q$. Let $Y_1(N)\hookrightarrow X_1(N)$ denote the canonical proper desingularisation of $Y_1(N)$ over $\Q$. As a Riemann surface over the complex numbers, it is obtained by adjoining the cusps. The modular curve $X_1(N)$ represents isomorphism classes of generalised elliptic curves with $\Gamma_1(N)$-level structure. Let $\pi: \cE\lra X_1(N)$ denote the universal generalised elliptic curve equipped with its canonical $\Gamma_1(N)$-level structure, and write $W_k$ for the canonical proper desingularisation of the $k$-fold self fibre product of $\cE$ over $\Q$ (see \cite[\S 3.0]{scholl} in the original case of full level $\Gamma(N)$, and \cite[Appendix]{bdp1} in the case of level $\Gamma_1(N)$, even over $\spec \Z[1/N]$). This is the $k$-th Kuga--Sato variety. It is smooth and proper over $\Q$ of dimension $k+1$, and has a natural fibration $\pi_k : W_k\lra X_1(N)$. The fibre over a non-cuspidal point $x$ representing the isomorphism class of an elliptic curve $E_x$ with $\Gamma_1(N)$-level structure is given by $\pi_k^{-1}(x)=E_x^k$. 

Scholl has constructed a projector $\epsilon_{W_k}$, which cuts out the space $S_{k+2}(\Gamma_1(N))$ of cusp forms of weight $k+2$ and level $\Gamma_1(N)$ inside the de Rham cohomology of the variety $W_k$ (see \cite[\S 1.1.2]{scholl} for the original construction of Scholl in full level, and \cite[(2.1.2)]{bdp1} for the case of level $\Gamma_1(N)$). We briefly recall its definition. Translation by the section of order $N$ of $\pi : \cE \lra X_1(N)$ given by the canonical $\Gamma_1(N)$-level structure gives rise to an action of $\Z/N\Z$ on $\cE$. Multiplication by $-1$ in the fibres of $\pi$ defines an action of $\mu_2$ on $\cE$. The symmetric group $\Sigma_k$ acts on $\cE^k$ by permuting the factors. There is therefore an action of the group 
\[
\Lambda_k:=(\Z/N\Z \rtimes \mu_2)^k \rtimes \Sigma_k
\]
on $\cE^k$. By the canonical nature of the desingularisation, this action extends to an action on the Kuga--Sato variety $W_k$ \cite[Theorem 3.1.0 (i)]{scholl}. Let $\chi_k : \Lambda_k \lra \{ \pm 1 \}$ be the character which is trivial on $(\Z/N\Z)^k$, the product character on $(\mu_2)^k$, and the sign character on $\Sigma_k$. The projector
\[ 
\epsilon_{W_k}:=\frac{1}{\vert \Lambda_k\vert} \sum_{g\in \Lambda_k} \chi_k(g)g \in \Z\left[\frac{1}{2N\cdot k !}\right][\Lambda_k] 
\]
is the one corresponding to the character $\chi_k$. We will view this projector as an idempotent element of $\corr^0(W_k, W_k)_\Q$ as follows. Given $g\in \Lambda_k$, denote by $\delta_g$ the induced automorphism of $W_k$, and let $\Gamma_{\delta_g}\subset W_k\times W_k$ denote its graph. By slight abuse of notation we define the idempotent correspondence
\begin{equation}\label{def:eps}
\epsilon_{W_k}:=\frac{1}{\vert \Lambda_k\vert} \sum_{g\in \Lambda_k} \chi_k(g)\Gamma_{\delta_g} \in \corr^0(W_k, W_k)_\Q.
\end{equation}
As such, $\epsilon_{W_k}$ acts on the various cohomology groups associated to $W_k$. The symmetry of the correspondence $\epsilon_{W_k}$ implies that the push-forward and pull-back maps it induces are equal, and we will denote any such map simply by $\epsilon_{W_k}$. 
 
\begin{proposition}\label{prop:epsiloncusp}
We have $\epsilon_{W_k} H_{\dR}^*(W_{k,\C})=\epsilon_{W_k} H_{\dR}^{k+1}(W_{k,\C}),$ and the Hodge filtration is given by $\Fil^{j} \epsilon_{W_k} H^{k+1}_{\dR}(W_{k,\C})=0$ for $j\geq k+2$, and 
\begin{equation*}
S_{k+2}(\Gamma_1(N))\simeq \Fil^{j} \epsilon_{W_k} H^{k+1}_{\dR}(W_{k,\C}),
\end{equation*}
for $1\leq j \leq k+1$,
via the association $f \mapsto \omega_f$ where $\omega_f(\C/\langle 1,\tau\rangle, 1/N):=f(\tau)(2\pi i dw)^{k}\otimes (2\pi i d\tau)$ for $\tau$ in the complex upper half-plane $\cH$ and $w$ the standard coordinate on the torus $\C/\langle 1,\tau\rangle$ with lattice $\langle 1,\tau\rangle:=\Z \oplus \Z\tau$.
\end{proposition}

\begin{proof}
This is \cite[Lemma 2.2 \& Corollary 2.3]{bdp1} (which hold more generally over any field of characteristic zero).
\end{proof}

\subsection{Products of Kuga--Sato varieties with powers of CM elliptic curves}\label{s:products}

Fix an imaginary quadratic field $K$ with ring of integers $\cO_K$ and discriminant $-d_K$ coprime to $N$. 
%For simplicity it will be assumed throughout that $d_K\neq 3,4$. This simplifying hypothesis implies that $\oh_K^\times=\{ \pm 1 \}$. 
Let $H$ be the Hilbert class field of $K$. By our conventions in Section \ref{s:convention}, a complex embedding $H\hookrightarrow \C$ is fixed. Let $A$ be an elliptic curve defined over $H$ with ring of endomorphisms $\End_H(A)$ isomorphic to $\oh_K$. Such an elliptic curve is said to have CM over $H$ by the maximal order $\oh_K$ of $K$. 

Let $k'$ be another even integer with $k'=2r'$ for some $r'\geq 0$. We will assume that $k\geq k'$.
Consider the variety $X_{k, k'}:=W_{k,H} \times_H A^{k'}$, which is smooth and proper over $H$ of dimension $k+k'+1$. It comes equipped with a natural fibration $\tilde{\pi}_{k, k'} : X_{k, k'} \lra X_1(N)$, whose fibre over a non-cuspidal point $x$ representing the isomorphism class of an elliptic curve $E_x$ with $\Gamma_1(N)$-level structure is given by $\tilde{\pi}_{k, k'}^{-1}(x)=E_x^k \times A^{k'}$. 

The group $\mu_2$ acts on $A$ by multiplication by $-1$, and the symmetric group $\Sigma_{k'}$ acts on $A^{k'}$ by permuting the factors. Hence $A^{k'}$ carries a natural action of the group $\Lambda'_{k'}:=(\mu_2)^{k'} \rtimes \Sigma_{k'}$. Let $\chi'_{k'} : \Lambda'_{k'}\lra \{ \pm 1 \}$ be the product character on $(\mu_2)^{k'}$ and the sign character on $\Sigma_{k'}$. Let 
$$
\epsilon_{A^{k'}}:=\frac{1}{\vert\Lambda'_{k'}\vert} \sum_{h\in \Lambda'_{k'}} \chi'_{k'}(h) h\in \Z\left[\frac{1}{2\cdot k !}\right][\Lambda'_{k'}] 
$$
be the projector associated with $\chi'_{k'}$. Given $h\in \Lambda'_k$, denote by $\delta'_h$ the induced automorphism of $A^{k'}$ and by $\Gamma_{\delta'_h}\subset A^{k'} \times A^{k'}$ its graph. As in the previous section, we denote by $\epsilon_{A^{k'}}$ the corresponding idempotent element of $\corr^0(A^{k'}, A^{k'})_{\Q}$ constructed using these graphs. 

We can now define the idempotent correspondence 
\begin{equation}\label{projector}
\epsilon_{X_{k, k'}}:=\epsilon_{W_k}\times \epsilon_{A^{k'}}:=\pr_{13}^*(\epsilon_{W_k})\cdot \pr_{24}^*(\epsilon_{A^{k'}})\in \corr^0(X_{k, k'}, X_{k, k'})_{\Q},
\end{equation}
where $\pr_{13} : X_{k, k'}^2\lra W_k^2$ and $\pr_{24}: X_{k, k'}^2\lra (A^{k'})^2$ are the natural projections. Explicitly, given $(g,h) \in \tilde{\Lambda}_{k,k'}:=\Lambda_k\times \Lambda'_{k'}$, denote by $\tilde{\delta}_{g,h}$ the automorphism $\delta_g \times \delta'_h$. Letting $\tilde{\chi}_{k,k'}$ denote the product character $\chi_k\times \chi'_{k'} : \tilde{\Lambda}_{k,k'}\lra \{\pm 1\}$, we have 
\[
\epsilon_{X_{k ,k'}}=\frac{1}{\vert \tilde{\Lambda}_{k,k'} \vert} \sum_{(g, h)\in \tilde{\Lambda}_{k,k'}} \tilde{\chi}_{k,k'}(g,h) \Gamma_{\tilde{\delta}_{g,h}}\in \corr^0(X_{k,k'}, X_{k,k'})_{\Q}.
\] 
Note the symmetry of this correspondence, in the sense that its induced push-forward and pull-back maps on Chow groups and cohomology groups are equal. We will therefore denote these simply by $\epsilon_{X_{k,k'}}$ by slight abuse of notation. 

\begin{notation}\label{notation}
In order to lighten the notation in the case $k'=k$, we will replace the subscript $k,k'$ simply by $k$, e.g., we will write $X_{k}$ for $X_{k,k}$. This convention will be adopted throughout the article.
\end{notation}

%The following result generalises Proposition \ref{prop:epsiloncusp}.

\begin{proposition}\label{eqn:filr-para-recall}
We have $\epsilon_{X_{k,k'}} H_{\dR}^*(X_{k,k',\C})=\epsilon_{X_{k,k'}} H_{\dR}^{k+k'+1}(X_{k,k',\C}),$ and the $(r+r'+1)$-th step of the Hodge filtration is identified with
\begin{equation}\label{iso:cuspX2}
S_{k+2}(\Gamma_1(N))\otimes \sym^{k'} H_{\dR}^1(A_\C) \simeq \Fil^{r+r'+1} \epsilon_{X_{k,k'}} H_{\dR}^{k+k'+1}(X_{k,k',\C}),
\end{equation}
via the assignment $f\otimes \alpha\mapsto \omega_f\wedge \alpha$ under the K\"unneth decomposition.
\end{proposition}

\begin{proof}
The case $k'=0$ is Proposition \ref{prop:epsiloncusp}.
The case $k'=k$ is \cite[Propositions 2.4 \& 2.5]{bdp1}. 
In general, we have $\epsilon_{A^{k'}} H_{\dR}^*(A^{k'}_\C)=\sym^{k'} H^1_{\dR}(A_\C)$ by \cite[Lemma 1.8]{bdp1}. The result then follows by Proposition \ref{prop:epsiloncusp} from the K\"unneth decomposition.
\end{proof}

\section{Algebraic cycles in CM fibres}\label{s:cycles}

Let $N\geq 5$ and $k=2r\geq k'=2r'\geq 0$ with $r\geq 1$. Fix an imaginary quadratic field $K$ satisfying the Heegner hypothesis with respect to $N$: every prime dividing $N$ splits in $K$. Let $A$ be an elliptic curve with CM by $\oh_K$ defined over the Hilbert class field $H$ of $K$.
By the Heegner hypothesis, there exists an ideal $\cN$ of $\oh_K$ such that $\oh_K/\cN = \Z/N\Z$. Fix such a choice of ideal $\cN$. Choose a generator $t$ of the cyclic group $A[\cN]$. Then the isomorphism class of $(A, t)$ is represented by a point $P_1$ in $Y_1(N)$ defined over the ray class field $K_{\cN}$ of $K$ of conductor $\cN$ by the theory of complex multiplication \cite[Theorem 11.39]{cox}. This point maps to the Heegner point $(A,A[\cN])$ in $Y_0(N)$ (as defined in \cite{grossX0}).

\subsection{Isogenies}\label{s:isog}

The algebraic cycles that we will consider on $X_{k, k'}=W_{k,H}\times_H A^{k'}$ are indexed by the set $\Isog^{\cN}(A)$ consisting of $\bar K$-isomorphism classes $(\varphi, A')$ of isogenies of elliptic curves $\varphi: A\lra A'$ defined over $\bar K$ whose kernels intersect $A[\cN]$ trivially. The Galois group $\Gal(\bar K / H)$ acts naturally on $\Isog^{\cN}(A)$, and $(\varphi, A')$ admits a representative defined over some field $H\subset F\subset \bar K$ if it is fixed by $\Gal(\bar K / F)$.

Any isogeny $\varphi : A \lra A'$ induces an isomorphism $K=\End(A)\otimes \Q\simeq \End(A')\otimes \Q$, and in particular the elliptic curve $A'$ has CM by some order $\oh_\varphi$ in $K$. Such orders are determined by their conductor $c_\varphi:=[\oh_K : \oh_\varphi]$. Given $c\in \mathbb{N}$, the unique order of conductor $c$ will be denoted by $\oh_c:=\Z+c\oh_K$, and we let $\Isog_{c}^{\cN}(A)$ denote the subset of $\Isog^{\cN}(A)$ consisting of those isomorphism classes $(\varphi, A')$ for which $A'$ has CM by $\oh_c$. 
%Denote $c_\varphi$ by  the conductor of $\oh_{\varphi}$, so that the discriminant of $\oh_\varphi$ is $-c_{\varphi}^2 d_K$. 
By the theory of complex multiplication \cite[Theorem 11.1]{cox}, a representative of $(\varphi, A')\in \Isog_c^{\cN}(A)$ can be taken to be defined, along with its complex multiplication, over the ring class field $H_{c}$ of $K$ of conductor $c$.
We then always fix the isomorphism $\End_{H_c}(A')\simeq \oh_c$ by the convention that $[\alpha]^* \omega' = \alpha \omega'$ for any regular differential form $\omega'$ of $A'$, where $[\alpha]$ denotes the element $\alpha\in \oh_c$ viewed as an endomorphism of $A'$. Note that $\oh_{\varphi}=\oh_{c_\varphi}$ always contains the order $\oh_{d_{\varphi}}$, where $d_\varphi$ is the degree of $\varphi$. 

\subsection{Variants of generalised Heegner cycles}\label{s:HC}

Heegner cycles are certain algebraic cycles on $X_{k,0,\bar\Q}=W_{k,\bar\Q}$ of codimension $r+1$, while the generalised Heegner cycles of \cite[\S 2.3]{bdp1} are cycles on $X_k=X_{k, k, \bar\Q}$ of codimension $k+1$. Variants of generalised Heegner cycles, as introduced in \cite[\S 4.1]{bdp3}, are algebraic cycles on $X_{k,k', \bar\Q}$ of codimension $r+r'+1$. We will now recall their definition. When $k'=0$, this gives the definition of Heegner cycles, while the case $k'=k$ recovers the definition of generalised Heegner cycles.

Given $(\varphi, A')\in \Isog^{\cN}(A)$, the isomorphism class of the pair $(A', \varphi(t))$ is represented by a rational point $P_{\varphi}$ in $Y_1(N)$, which determines an embedding $\iota'_{\varphi}$ of $(A')^{k}$ in $W_k$ seen as the fibre $\pi_k^{-1}(P_\varphi)$. This in turn determines an embedding $\iota_{\varphi}=\iota_\varphi' \times \id_{A^{k'}}$ of $(A')^{k}\times A^{k'}$ in $X_{k,k'}$ seen as the fibre $\tilde{\pi}_{k, k'}^{-1}(P_{\varphi})$.
Consider the graph $\Gamma_{[d_{\varphi}\sqrt{-d_K}]}\subset A'\times A'$ of the endomorphism $[d_{\varphi}\sqrt{-d_K}]\in \End(A')$, as well as the graph $\Gamma_{\varphi}\subset A\times A'$ of $\varphi$. Define 
\[
\Upsilon_{k,k',\varphi}:= (\Gamma_{\varphi})^{k'} \times (\Gamma_{[d_{\varphi}\sqrt{-d_K}]})^{r-r'} \subset (A\times A')^{k'}\times (A'\times A')^{r-r'}=(A')^{k}\times A^{k'} \overset{\iota'_{\varphi}}{\subset} X_{k, k'}.
\]
Applying the projector \eqref{projector} gives rise to the variant of the generalised Heegner cycle associated to $\varphi$
\[
\Delta_{k,k',\varphi}:=\epsilon_{X_{k,k'}}\Upsilon_{k,k',\varphi} \in \CH^{r+r'+1}(X_{k,k', \bar \Q})_{0, \Q}
\]
in the Chow group of codimension $r+r'+1$ cycles on $X_{k,k', \bar \Q}$ with rational coefficients. The cycle $\Delta_{k,k',\varphi}$ is null-homologous since cycle class maps are functorial with respect to correspondences and $\epsilon_{X_{k,k'}}$ annihilates the target of the cycle class map by Proposition \ref{eqn:filr-para-recall}. That the cycles are defined over $\bar \Q$ (in fact over $K^{\ab}$) follows from the following:

\begin{proposition}\label{prop:fod}
Let $c\in \mathbb{N}$. If $(\varphi, A')\in \Isog_{c}^{\cN}(A)$, then the cycle $\Delta_{k,k',\varphi}$ is defined over the field compositum $F_{c} := K_{\cN}\cdot H_{c}\subset K^{\ab}\subset \bar \Q$.
\end{proposition}

\begin{proof}
The elliptic curve $A$ was chosen to be defined over the Hilbert class field $H$ of $K$ (possible since $H=K(j(A))$ by \cite[Theorem 11.1]{cox}). Similarly, we may choose a representative of the isomorphism class $(\varphi, A')\in \Isog_c^{\cN}(A)$ that is defined over $H_{c}$ (again by \cite[Theorem 11.1]{cox}). The fixed $\Gamma_1(N)$-level structure $t\in A[\cN]$ is defined over the extension $H(A[\cN])$ obtained by adjoining the coordinates of the $\cN$-torsion points. This extension is abelian over $H$, but not necessarily over $K$. However, the isomorphism class of $(A,t)$ as an elliptic curve with $\Gamma_1(N)$-level structure (i.e., the point $P_1$ of $Y_1(N)$) is defined over the ray class field $K_{\cN}=H(h_A(A[\cN]))$, where $h_A : A\lra A/\Aut(A)=A/\oh_K^\times$ is the Weber function \cite[Theorem 11.39]{cox}. It follows that the isomorphism class of $(A', \varphi(t))$ (i.e., the point $P_\varphi$ of $Y_1(N)$) is defined over $F_{c}$, hence so is the embedding $\iota'_{\varphi}$. Finally, $\End_{H_{c}}(A')\simeq \oh_c$ and thus $[d_\varphi\sqrt{-d_K}]$ is defined over $H_{c}$. This shows that $\Upsilon_{k,k',\varphi}$ is defined over $F_c$. Since $\epsilon_{k,k'}$ is defined over $\Q$, the result follows.
\end{proof}

\begin{definition}\label{rem:hc}
Let $c\in \mathbb{N}$ and $(\varphi, A')\in \Isog^{\cN}_c(A)$.
\begin{itemize}
\item When $k'=0$ in the above construction, we write $\Upsilon_{\varphi}^{\HC}:=\Upsilon_{k,0,\varphi}$ and 
\[
\Delta_{\varphi}^{\HC}:=\Delta_{k,0, \varphi}=\epsilon_{W_k}\Upsilon_{\varphi}^{\HC}\in \CH^{r+1}(W_{k, F_{c}})_{0, \Q}.
\]
This is the Heegner cycle associated to $\varphi$ and studied for instance in \cite[\S 5]{NekovarKS}.
\item When $k'=k$ in the above construction, we write $\Upsilon_{\varphi}^{\GHC}:=\Upsilon_{k,k,\varphi}=\Upsilon_{k,\varphi}$ and 
\[
\Delta_{\varphi}^{\GHC}:=\Delta_{k,k, \varphi}=\Delta_{k, \varphi}=\epsilon_{X_k}\Upsilon_{\varphi}^{\GHC}\in \CH^{k+1}(X_{k,F_{c}})_{0, \Q}.
\]
This is the generalised Heegner cycle associated to $\varphi$ first introduced in \cite[\S 2]{bdp1}.
\end{itemize}
\end{definition}

The cycles introduced so far are elements of Chow groups with {\em rational coefficients}.
In order to meaningfully consider their images under Abel--Jacobi maps and discuss their torsion or non-torsion properties, we clear denominators by multiplying the cycles by suitable integers. 
\begin{definition}\label{rem:intcycles}
Let $m_{k,k'}:=\vert \tilde{\Lambda}_{k,k'} \vert = (2N)^k k! 2^{k'} (k')!$. Define the correspondence $$\tilde{\epsilon}_{X_{k,k'}}:=m_{k,k'}\epsilon_{X_{k,k'}}\in \corr^0(X_{k,k'}, X_{k,k'}),$$ and for all $(\varphi, A')\in \Isog^{\cN}(A)$ define the cycles
\[
\tilde{\Delta}_{k,k',\varphi}:=m_{k,k'}\Delta_{k,k',\varphi}=\tilde\epsilon_{X_{k,k'}}\Upsilon_{k,k',\varphi}\in  \CH^{r+r'+1}(X_{k,k', \bar \Q})_{0}.
\]
These cycles have {\em integral coefficients} and inherit the properties described in Proposition \ref{prop:fod}. 
\end{definition}

%\subsection{Generalised Heegner cycles}

%Generalised Heegner cycles are certain algebraic cycles on $X_k$ of codimension $k+1$, which are indexed by $\Isog^{\cN}(A)$. Given an isogeny $\varphi : A \lra A'$ in this set, the point $P_\varphi$ in $Y_1(N)$ determines the embedding $$\iota_{\varphi}:=\iota'_{\varphi}\times \id : (A')^k\times A^k \hookrightarrow W_k\times A^k=X_k$$ of $(A')^k\times A^k$ seen as the fibre of $\tilde{\pi}_k$ above $P_{\varphi}$. By taking the graph $\Gamma_{\varphi}\subset A\times A'$ of $\varphi$, it is possible to form the $k$-dimensional algebraic cycle 
%\[
%\Upsilon^{\GHC}_{\varphi}:=(\Gamma_{\varphi})^k \subset (A\times A')^k \simeq (A')^k\times A^k \overset{\iota_{\varphi}}{\subset} X_k.
%\]
%Applying the projector $\epsilon_{X_k}$, and considering the image of the resulting cycle modulo rational equivalence gives rise to the generalised Heegner cycle associated to $\varphi$
%\[
%\Delta^{\GHC}_{\varphi}:=\epsilon_{X_k} \Upsilon^{\GHC}_\varphi \in \CH^{k+1}(X_k)_{0, \Q}
%\]
%in the Chow group of codimension $k+1$ cycles in $X_k$. The cycle $\Delta^{\GHC}_{\varphi}$ is null-homologous since cycle class maps are functorial with respect to correspondences, and $\epsilon_{X_k}$ annihilates the target of the cycle class map by Proposition \ref{eqn:filr-para-recall}. Furthermore, it is similarly defined over the field compositum $F_{\varphi}$.

\subsection{Relation with Heegner cycles}\label{s:relation}

In \cite[\S 2.4]{bdp1}, Bertolini, Darmon, and Prasanna exhibited a correspondence from $X_{k}$ to $W_k$ mapping generalised Heegner cycles to multiples of Heegner cycles. The details of this calculation were left to the reader. A more general setup was worked out by the same authors in \cite[Proposition 4.1.1]{bdp3}. They exhibited a correspondence from $X_k$ to $X_{k,k'}$ mapping generalised Heegner cycles to multiples of the cycles $\Delta_{k,k', \varphi}$. This was done for specific isogenies between elliptic curves both having CM by $\oh_K$. In this section, we exhibit a correspondence from $X_{k,k'}$ to $W_k$, which maps $\Delta_{k,k', \varphi}$ to an integer multiple of $\Delta_\varphi^{\HC}$ for all $(\varphi, A')\in \Isog^{\cN}(A)$.

Consider the variety $W_{k,H}\times_H A^{r'}$ embedded into $Z_{k,k'}:=X_{k,k'}\times_H W_{k,H}=W_{k,H}\times_H A^{k'} \times_H W_{k,H}$ via the map $\Psi_{k,k',\varphi}:=(\id_{W_k}, (\id_A, [d_\varphi \sqrt{-d_K}])^{r'}, \id_{W_k})$. Denote its image by $\Pi_{k,k',\varphi}$. This is a $k+r'+1$ dimensional subvariety of the variety $Z_{k,k'}$ of dimension $2k+k'+2$. Its class modulo rational equivalence therefore gives rise to a correspondence 
\begin{equation}\label{Pikk}
\Pi_{k,k',\varphi} \in \CH^{k+k'+1-r'}(X_{k,k'}\times_H W_{k,H})=\corr^{-r'}(X_{k,k'}, W_{k,H})
\end{equation}
defined over $H$. 
This in turn induces push-forward and pull-back maps on Chow groups and cohomology groups in the usual way. In particular, it induces via push-forward a map 
\[
(\Pi_{k,k',\varphi})_* : \CH^{r+r'+1}(X_{k,k', \bar \Q})_{0,\Q} \lra \CH^{r+1}(W_{k,\bar \Q})_{0,\Q}. 
\]
We will use the following notations for the various natural projection maps:
\begin{equation}\label{diag1}
\begin{tikzcd}
 Z_{k,k'} 
 & =
 & X_{k,k'} \times_H W_{k,H} \arrow{dl}[swap]{\pi_{01}} \arrow{dr}{\pi_2}
 & =
 & W_{k,H} \times_H A^{k'} \times_H W_{k,H} \arrow{dl}[swap]{\pi_0} \arrow{d}{\pi_1} \arrow{dr}{\pi_2}
 & \\
 & X_{k,k'} 
 &
 & W_{k,H}
 & A^{k'} 
 & W_{k,H}.
\end{tikzcd}
\end{equation}

\begin{proposition}\label{prop:relation}
Let $(\varphi, A')\in \Isog^{\cN}(A)$ of degree $d_\varphi$. Then 
\[
(\Pi_{k,k',\varphi})_*(\Upsilon_{k,k',\varphi})=d_\varphi^{r'} \Upsilon_{\varphi}^{\HC} \quad \text{ and } \quad (\Pi_{k,k',\varphi})_*(\Delta_{k,k',\varphi})=d_\varphi^{r'} \Delta_{\varphi}^{\HC}.
\]
\end{proposition}

\begin{proof}
By definition of the push-forward map, we have 
$$(\Pi_{k,k',\varphi})_*(\Upsilon_{k,k',\varphi})=(\pi_2)_*(\Pi_{k,k',\varphi} \cdot (\pi_{01})^*(\Upsilon_{k,k',\varphi})),$$ 
where $\cdot$ denotes the intersection product in the Chow ring of $Z_{k,k'}$. Note that $(\pi_{01})^*(\Upsilon_{k,k',\varphi})$ is described by
\[
\{ (\iota'_\varphi((\varphi(x_i))_{i=1}^{k'}, (y_i, [d_\varphi\sqrt{-d_K}](y_i))_{i=1}^{r-r'}), (x_i)_{i=1}^{k'}, z) \: \vert \: (x_i)_{i=1}^{k'} \in A^{k'}, (y_i)_{i=1}^{r-r'}\in A^{r'}, z\in W_{k}  \}, 
\]
and by definition we have 
\[
\Pi_{k,k',\varphi} = \{ (s, (t_1, [d_{\varphi}\sqrt{-d_K}](t_1), \ldots, t_{r'}, [d_{\varphi}\sqrt{-d_K}](t_{r'})), s) \: \vert \: s\in W_{k}, (t_i)_{i=1}^{r'} \in A^{r'}  \}.
\]
The ambient variety $Z_{k,k'}$ is smooth of dimension $2k+k'+2$, and the two subvarieties $(\pi_{01})^*(\Upsilon_{k,k',\varphi})$ and $\Pi_{k,k',\varphi}$ have respective codimensions $r+r'+1$ and $k+r'+1$. Set theoretically we see that they intersect in a subvariety of dimension $r$. In particular, they are dimensionally transverse. By \cite[Proposition 1.28]{EisHar}, it follows that $(\pi_{01})^*(\Upsilon_{k,k',\varphi})$ and $\Pi_{k,k',\varphi}$ intersect generically transversely. By \cite[Theorem 1.26 (b)]{EisHar}, we deduce that $\Pi_{k,k',\varphi}\cdot (\pi_{01})^*(\Upsilon_{k,k',\varphi}) = \Pi_{k,k',\varphi} \cap (\pi_{01})^*(\Upsilon_{k,k',\varphi})$ in the Chow group, where $\cap$ denotes the set theoretic intersection.
%It follows that these two cycles intersect transversally, and we may compute their intersection product by taking a set theoretic intersection. 
Using this, we obtain that  $(\Pi_{k,k',\varphi})_*(\Upsilon_{k,k',\varphi})$ is given by
\[
d_{\varphi}^{r'} \{  \iota'_{\varphi}((\varphi(t_i), \varphi([d_{\varphi}\sqrt{-d_K}](t_i)))_{i=1}^{r'}, (y_i, [d_\varphi\sqrt{-d_K}](y_i))_{i=1}^{r-r'}) \: \vert \: (t_i)_{i=1}^{r'} \in A^{r'}, (y_i)_{i=1}^{r-r'}\in A^{r'} \},
\]
where the appearance of the degree of $\varphi$ stems from the push-forward by $\pi_2$. Observing that $\varphi \circ [d_{\varphi}\sqrt{-d_K}]=[d_{\varphi}\sqrt{-d_K}]\circ \varphi$ leads to $(\Pi_{k,k',\varphi})_*(\Upsilon_{k,k',\varphi})$ being equal to
\[
 d_{\varphi}^{r'} \{  \iota'_{\varphi}((\varphi(t_i), [d_{\varphi}\sqrt{-d_K}](\varphi(t_i)))_{i=1}^{r'}, (y_i, [d_\varphi\sqrt{-d_K}](y_i))_{i=1}^{r-r'}) \: \vert \: (t_i)_{i=1}^{r'} \in A^{r'}, (y_i)_{i=1}^{r-r'}\in A^{r'} \}.
\]
The latter is $d_{\varphi}^{r'}\Upsilon_{k, 0,\varphi}$ since $\varphi$ is surjective. We have proved that $(\Pi_{k,k',\varphi})_*(\Upsilon_{k,k',\varphi})=d_\varphi^{r'} \Upsilon_{\varphi}^{\HC}$.

In order to prove the equality $(\Pi_{k,k',\varphi})_*(\Delta_{k,k',\varphi})=d_\varphi^{r'} \Delta_{\varphi}^{\HC}$, it suffices to prove that 
\begin{equation}\label{eq:epsiloncommutes}
(\Pi_{k,k',\varphi} \circ \epsilon_{X_{k,k'}})_*(\Upsilon_{k,k',\varphi})=(\epsilon_{W_k}\circ \Pi_{k,k',\varphi})_*(\Upsilon_{k,k',\varphi})
\end{equation}
in $\CH^{r+1}(W_{k,\bar \Q})_{\Q}$, where $\circ$ denotes composition of correspondences. 
%Consider the subgroup $\Lambda'_k\times \Lambda'_k$ of $\Lambda_k\times \Lambda'_k$, the latter being the group appearing in the sum defining $\epsilon_{X_k}$. 
Given $g\in \Lambda_{k}$, 
%denote by $\delta_g$ the automorphism $\hat{\delta}_{g, 1}$ of $X_k$ given by the action of $g$ on the component $W_k$, and denote by $\tilde{\delta}_g$ the automorphism $\hat{\delta}_{1, h}$ of $X_k$ given by the action of $h$ on the component $A^k$. Note the slight abuse of notations. 
we observe by direct calculation that 
\begin{equation}\label{commute}
\Pi_{k,k',\varphi} \circ \Gamma_{\tilde{\delta}_{g, 1}}=\Gamma_{\delta_g} \circ \Pi_{k,k',\varphi}.
\end{equation}
Define a map $\alpha : \Lambda_{k'}'\lra \Lambda'_k$ as follows. Given $((\mu_1, \ldots, \mu_{k'}), \sigma)\in \Lambda_{k'}'=(\mu_2)^{k'}\rtimes \Sigma_{k'}$, define
\[
\alpha((\mu_1, \ldots, \mu_{k'}), \sigma)=((\mu_1, \ldots, \mu_{k'}), (1, \ldots, 1), \sigma)\in \Lambda'_k=((\mu_2)^{k'}\times (\mu_2)^{k-k'})\rtimes \Sigma_k,
\]
where the natural inclusion $\Sigma_{k'}\subset \Sigma_k$ is obtained by permuting the first $k'$ factors. The map $\alpha$ is an injective group homomorphism. Composing with the inclusion $\Lambda'_k\subset \Lambda_k\subset \tilde{\Lambda}_{k,k'}$ realises $\Lambda'_{k'}$ as a subgroup of $\tilde{\Lambda}_{k,k'}$, and thus also as a subgroup of the group of automorphisms of $X_{k,k'}$. Similarly, the map $\alpha$ composed with the inclusion $\Lambda'_{k}\subset \Lambda_k$ realises $\Lambda'_{k'}$ as a subgroup of the group of automorphisms of $W_k$. Consider also the injective group homomorphism $1\times \id : \Lambda_{k'}' \subset \tilde{\Lambda}_{k,k'}=\Lambda_k\times \Lambda'_{k'}$. The product of these two homomorphisms realises $(\Lambda_{k'}')^2$ as a subgroup of $\tilde{\Lambda}_{k,k'}$, and thus as a subgroup of the group of automorphisms of $X_{k,k'}$. Note that $\tilde{\Lambda}_{k,k'}/(\Lambda_{k'}')^2\simeq \Lambda_{k}/\alpha(\Lambda_{k'}')$.
%we denote by $\tilde{h}$ the image of $h$ under the map $\Lambda'_{k'}\lra \Lambda_k$ sending $(\mu, \sigma)$ in $(\mu_2)^{k'}\times \Sigma_{k'}$ to $(0, (\mu), (1)_{i=1}^{k-k'}, \sigma)$  in $(\Z/N\Z)^{k}\times (\mu_2)^{k'}\times (\mu_2)^{k-k'} \times \Sigma_k$, where $\sigma\in \Sigma_{k'}$ is viewed naturally inside $\Sigma_k$ by permuting the first $k'$ letters. The induced automorphism
It follows from \eqref{commute} that, for any $(g, h)\in (\Lambda'_{k'})^2$, we have
\[
\Pi_{k,k',\varphi} \circ \Gamma_{\tilde{\delta}_{\alpha(gh), h}} = \Pi_{k,k',\varphi} \circ \Gamma_{\tilde{\delta}_{\alpha(g), 1}}\circ \Gamma_{\tilde{\delta}_{\alpha(h),h}}=\Gamma_{\delta_{\alpha(g)}} \circ \Pi_{k,k',\varphi} \circ \Gamma_{\tilde{\delta}_{\alpha(h),h}}
\]
in $\corr^{-r}(X_k, W_{k,H})_\Q$. It is immediate that $(\Gamma_{\tilde{\delta}_{\alpha(h),h}})_* (\Upsilon_{k,k',\varphi})=\Upsilon_{k,k',\varphi}$, and therefore 
\begin{equation}\label{commute2}
(\Pi_{k,k',\varphi} \circ \Gamma_{\tilde{\delta}_{\alpha(gh), h}})_*(\Upsilon_{k,k',\varphi})=(\Gamma_{\delta_{\alpha(g)}} \circ \Pi_{k,k',\varphi})_*(\Upsilon_{k,k',\varphi}), \quad \text{ for all } (g,h) \in  (\Lambda_{k'}')^2.
\end{equation}
This is similar to the equality obtained in \cite[(4.1.4)]{bdp3}. The following calculation is inspired by the one at the end of the proof of \cite[Proposition 4.1.1]{bdp3}:
\begin{align*}
& (\Pi_{k,k',\varphi}\circ  \epsilon_{X_{k,k'}})_*(\Upsilon_{k,k',\varphi})
= (\Pi_{k,k',\varphi})_* \left( \frac{1}{\vert \tilde{\Lambda}_{k,k'}\vert} \sum_{(s,t)\in \tilde{\Lambda}_{k,k'}} \tilde{\chi}_{k,k'}(s, t)(\Gamma_{\tilde{\delta}_{s, t}})_*(\Upsilon_{k,k', \varphi})  \right) \\
&= (\Pi_{k,k',\varphi})_* \left( \frac{\vert \Lambda'_{k'}\vert^2}{\vert \tilde{\Lambda}_{k,k'} \vert \vert \Lambda'_{k'}\vert^2} \sum_{\substack{(s,1)\in \tilde{\Lambda}_{k,k'}/(\Lambda'_{k'})^2 \\ (g,h)\in(\Lambda'_{k'})^2}} \tilde{\chi}_{k,k'}(s\alpha(g), h)(\Gamma_{\tilde{\delta}_{s\alpha(g), h}})_*(\Upsilon_{k,k', \varphi})  \right) \\
&\overset{\eqref{commute}}{=} \frac{\vert \alpha(\Lambda_{k'}') \vert}{\vert \Lambda_{k} \vert} \sum_{s\in \Lambda_{k}/\alpha(\Lambda'_{k'})} \chi_k(s) (\Gamma_{\delta_s})_* (\Pi_{k,k',\varphi})_*\left( \frac{1}{\vert \Lambda'_{k'}\vert^2}\sum_{(g,h)\in(\Lambda'_{k'})^2} \tilde{\chi}_{k,k'}(\alpha(g), h)(\Gamma_{\tilde{\delta}_{\alpha(g), h}})_*(\Upsilon_{k,k', \varphi})  \right) \\
&= \frac{\vert \alpha(\Lambda_{k'}') \vert}{\vert \Lambda_{k} \vert} \sum_{s\in \Lambda_{k}/\alpha(\Lambda'_{k'})} \chi_k(s) (\Gamma_{\delta_s})_* (\Pi_{k,k',\varphi})_*\left( \frac{1}{\vert \Lambda'_{k'}\vert^2}\sum_{(g,h)\in(\Lambda'_{k'})^2} \chi_{k}(\alpha(g))(\Gamma_{\tilde{\delta}_{\alpha(gh), h}})_*(\Upsilon_{k,k', \varphi})  \right) \\
& \overset{\eqref{commute2}}{=}  \frac{\vert \alpha(\Lambda_{k'}') \vert}{\vert \Lambda_{k} \vert} \sum_{s\in \Lambda_{k}/\alpha(\Lambda'_{k'})} \chi_k(s) (\Gamma_{\delta_s})_* \left( \frac{1}{\vert \Lambda'_{k'}\vert}\sum_{g\in\Lambda'_{k'}} \chi_k(\alpha(g))(\Gamma_{\delta_{\alpha(g)}})_*(\Pi_{k,k',\varphi})_*(\Upsilon_{k,k', \varphi})  \right) \\
&= \frac{1}{\vert \Lambda_k \vert} \sum_{s\in \Lambda_k} \chi_k(s) (\Gamma_{\delta_s})_* (\Pi_{k,k',\varphi})_*(\Upsilon_{k,k', \varphi})   \\
& = (\epsilon_{W_k}\circ \Pi_{k,k',\varphi})_*(\Upsilon_{k,k', \varphi}). 
\end{align*}
In the fourth equality, we made the change of variables $g\leftrightarrow gh$, and used the fact that $$\tilde{\chi}_{k,k'}(\alpha(gh), h)=\chi_k(\alpha(g))\chi_k(\alpha(h))\chi'_{k'}(h)=\chi_k(\alpha(g)),$$ since $\chi_k\vert_{\alpha(\Lambda_{k'}')}=\chi'_{k'}$ and these characters are quadratic. 
\end{proof}

\section{Complex Abel--Jacobi maps}\label{s:caj}

Let $V$ denote a smooth projective variety of dimension $d$ defined over $\C$. The familiar Abel--Jacobi map for curves admits a higher dimensional analogue
\begin{equation}\label{map:CAJ}
\AJ_{V}^j : \CH^j(V)_0\lra J^j(V) :=\frac{(\Fil^{d-j+1} H^{2d-2j+1}_{\dR}(V))^\vee}{H_{2d-2j+1}(V(\C), \Z)},
\end{equation} 
defined by the integration formula
\[
\AJ_{V}^j(Z)(\beta)=\int_{\partial^{-1}(Z)} \beta, \qquad \text{for } \beta\in \Fil^{d-j+1} H^{2d-2j+1}_{\dR}(V),
\]
where $\partial^{-1}(Z)$ denotes any continuous $(2d-2j+1)$-chain in $V(\C)$ whose image under the boundary map $\partial$ is $Z$. Here, $H_{2d-2j+1}(V(\C), \Z)$ is seen as a lattice by taking its image in the space $(\Fil^{d-j+1} H^{2d-2j+1}_{\dR}(V))^\vee$ via integration of differential forms over topological chains.
The target of $\AJ_V^j$ is the $j$-th intermediate Jacobian of $V$, which by Poincar\'e duality can be identified with
\begin{equation}\label{PDinterjac}
J^j(V)\simeq H^{2j-1}(V(\C), \C)/(\Fil^{j} H_{\dR}^{2j-1}(V) \oplus H^{2j-1}(V(\C), \Z)).
\end{equation}

We are interested in the Abel--Jacobi maps of the varieties $W_{k,\C}=X_{k,0,\C}$ and $X_{k,\C}=X_{k,k,\C}$, and in particular in the images of Heegner cycles and generalised Heegner cycles.
Observe, using the notations of Definition \ref{rem:intcycles}, that $m_{k,0}\tilde{\Delta}_{\varphi}^{\HC}=\tilde{\epsilon}_{W_{k}}\tilde{\Delta}_{\varphi}^{\HC}$ and $m_{k,k}\tilde{\Delta}_{\varphi}^{\GHC}=\tilde{\epsilon}_{X_{k}}\tilde{\Delta}_{\varphi}^{\GHC}$.
Since Abel--Jacobi maps are functorial with respect to correspondences \cite[Propositions 1,2 \& 4 iii)]{zucker}, we will solely be interested in the pieces of these maps that survive after applying the relevant correspondences. 
%Since these cycles are defined using the correspondences $\tilde{\epsilon}_{W_k}$ and $\tilde{\epsilon}_{X_k}$ respectively (see Remark \ref{rem:intcycles}), and Abel--Jacobi maps are functorial with respect to correspondences \cite[Propositions 1,2 \& 4 iii)]{zucker}, we will solely be interested in the pieces of these maps that survive after applying the relevant projectors. 
%observe that $\AJ^{r+1}_{W_k}(m_{k,0}\tilde{\Delta}_{\varphi}^{\HC})=\AJ^{r+1}_{W_{k}}(\tilde{\epsilon}_{W_{k}}\tilde{\Delta}_{\varphi}^{\HC})$ and $\AJ^{k+1}_{X_k}(m_{k,k}\tilde{\Delta}_{\varphi}^{\GHC})=\AJ^{k+1}_{X_{k}}(\tilde{\epsilon}_{X_{k}}\tilde{\Delta}_{\varphi}^{\GHC})$.
%Since Abel--Jacobi maps are functorial with respect to correspondences \cite[Propositions 1,2 \& 4 iii)]{zucker}, we will solely be interested in the pieces of these maps that survive after applying the relevant projectors. 
By Propositions \ref{prop:epsiloncusp} and \ref{eqn:filr-para-recall} (with $k'=k$) respectively, this allows us to view the relevant Abel--Jacobi maps as homomorphisms
\begin{equation}\label{epsAJ}
\AJ_{W_k} := \AJ_{W_k}^{r+1} \circ \tilde{\epsilon}_{W_k} = \tilde{\epsilon}_{W_k} \circ  \AJ_{W_k}^{r+1} : \CH^{r+1}(W_{k,\C})_0 \lra \frac{S_{k+2}(\Gamma_1(N))^\vee}{L_k}, 
\end{equation}
where $L_k=\tilde{\epsilon}_{W_k} H_{k+1}(W_{k,\C}(\C), \Z)$, and 
\[
\AJ_{X_k} := \AJ_{X_k}^{k+1} \circ \tilde{\epsilon}_{X_k} = \tilde{\epsilon}_{X_k} \circ  \AJ_{X_k}^{k+1} : \CH^{k+1}(X_{k,\C})_0 \lra \frac{(S_{k+2}(\Gamma_1(N))\otimes \sym^k H_{\dR}^1(A_\C))^\vee}{\tilde{L}_{k}}, 
\]
where $\tilde{L}_{k}=\tilde{\epsilon}_{X_k} H_{2k+1}(X_{k,\C}(\C), \Z)$. 
With these notations, for all $(\varphi, A')\in \Isog^{\cN}(A)$, we have 
\begin{equation}\label{eq:aj}
\AJ^{r+1}_{W_k}(m_{k,0}\tilde{\Delta}_{\varphi}^{\HC})=\AJ_{W_{k}}(\tilde{\Delta}_{\varphi}^{\HC}) \qquad \text{ and } \qquad \AJ^{k+1}_{X_k}(m_{k,k}\tilde{\Delta}_{\varphi}^{\GHC})=\AJ_{X_{k}}(\tilde{\Delta}_{\varphi}^{\GHC}).
\end{equation}

The Abel--Jacobi formula that we are about to state gives an expression for $\AJ_{W_k}(\tilde{\Delta}_{\varphi}^{\HC})$ in $S_{k+2}(\Gamma_1(N))^\vee$ modulo a lattice $L'_k$ which is slightly larger that $L_k$. This is less precise, but the resulting formula gains in explicitness. 
\begin{definition}
Define the lattice $L'_k\subset S_{k+2}(\Gamma_1(N))^\vee$ to be the $\Z$-module
generated by the period lattice attached to $S_{k+2}(\Gamma_1(N))$ (see \cite[Definition 3]{bdlp}) and the functionals $J_{s,t,P}$ defined by
\[
J_{\alpha,\beta,P}(f) := (2\pi i )^{k+1} \int_\alpha^\beta P(z) f(z) dz, \qquad
\mbox{ with } \alpha,\beta\in \PP^1(\Q),\quad P(X)\in \Z[X]^{\deg=k}.
\]
See \cite[\S 9]{bdlp} for further details. 
\end{definition}

The goal of this section is to prove the following:
\begin{theorem}\label{thm:CAJ}
Let $N\geq 5$ and $k=2r\geq 2$ be integers.
Let $K$ be an imaginary quadratic field satisfying the Heegner hypothesis with respect to $N$, and fix a choice of cyclic $N$-ideal $\cN$. Let $A$ be an elliptic curve with CM by $\oh_K$ over the Hilbert class field $H$ with a $\Gamma_1(N)$-level structure $t\in A[\cN]$. Let $\varphi : A_\C \lra \C/\langle 1, \tau' \rangle$ be an isogeny of degree $d_\varphi$ representing an element of $\Isog^{\cN}(A)$ and satisfying $\varphi(t)=\frac{1}{N} \pmod{\langle 1, \tau' \rangle}$. 
 Then, for all $f\in S_{k+2}(\Gamma_1(N))$, we have 
 \[
 \AJ_{W_k}((2^k k!)^2 d_{\varphi}^{r}\tilde{\Delta}_{\varphi}^{\HC})(\omega_f)=\frac{(-2\sqrt{-d_K})^r d_{\varphi}^{k} (2\pi i)^{r+1}m_{k,k}^2}{(\tau'-\overline{\tau}')^r} \int_{i\infty}^{\tau'}(z-\tau')^r(z-\bar{\tau}')^r f(z) dz \pmod {L'_k}.
 \]
 \end{theorem}
 
 \begin{remark}
A formula for the images of generalised Heegner cycles under the complex Abel--Jacobi map $\AJ^{k+1}_{X_k}$ was established in joint work of the author with Bertolini, Darmon, and Prasanna \cite[Theorem 1]{bdlp} by writing down explicit bounding chains for generalised Heegner cycles and computing the defining integrals in terms of explicit line integrals of modular forms. It was noted in \cite[Remark 10]{bdlp} that the techniques used can likely be adapted to calculate the images of Heegner cycles under the complex Abel--Jacobi map $\AJ^{r+1}_{W_k}$. While this is indeed possible, we have opted for a different method. We will use Proposition \ref{prop:relation} together with the functorial properties of Abel--Jacobi maps to deduce the formula for Heegner cycles directly from the formula for generalised Heegner cycles.
\end{remark}

 \begin{remark}\label{rem:hopkins}
Fix a normalised newform $f$ in $S_{k+2}(\Gamma_1(N))$. Composing the complex Abel--Jacobi map \eqref{epsAJ} with the projection to the $f$-isotypical Hecke component of the intermediate Jacobian gives rise to a map $\AJ_{W_k, f} : \CH^{r+1}(W_{k,\C})_0\lra \C/L_{f}$, where $L_f$ is the period lattice of $f$ in $\C$. Theorem \ref{thm:CAJ} yields the formula 
\begin{equation}\label{hopkins}
\AJ_{W_k, f}((2^k k!)^2 d_{\varphi}^{r}\tilde{\Delta}_\varphi^{\HC})=\frac{(-2\sqrt{-d_K})^r d_{\varphi}^{k}m_{k,k}^2}{(\tau'-\overline{\tau}')^r}\alpha(\tau') \pmod {L_f},
\end{equation}
where $\alpha : \Gamma_1(N)\backslash \mathcal{H}^{\mathrm{CM}}\lra \C/L_f$ is the map on Heegner points of \cite[Lemma 2.2]{hopkins} precomposed with $\Gamma_1(N)\backslash \mathcal{H}^{\mathrm{CM}}\twoheadrightarrow  \Gamma_0(N)\backslash \mathcal{H}^{\mathrm{CM}}$. Suppose that the fixed embedding $H\hookrightarrow \C$ of Section \ref{s:convention} is such that $A_\C=\C/\oh_K=\C/\langle 1, \tau\rangle$ with $\tau =(-d_K+\sqrt{-d_K})/2 \in \cH$ the standard generator of $\oh_K$. Suppose also that $\tau'=d_{\varphi}\tau$ (e.g., take $\varphi : \C/\langle 1, \tau\rangle\lra \C/\langle 1/n, \tau\rangle\overset{[n]}{\lra} \C/\langle 1, n\tau\rangle$ for some positive integer $n$). In this case, formula \eqref{hopkins} becomes
\[
\AJ_{W_k, f}((2^k k!)^2 d_{\varphi}^{r}\Delta_\varphi^{\HC})=(-2d_\varphi)^r m_{k,k}^2\alpha(\tau') \pmod {L_f}.
\]
As explained in \cite{hopkins}, a relation such as \eqref{hopkins} was expected to hold, but was not verified except for weight $4$ newforms (i.e., when $k=2$) as a consequence of the work of Schoen \cite{schoen}. The relation \eqref{hopkins} implies the compatibility of the conjectural partial generalisations of the Gross--Kohnen--Zagier theorem for higher weights formulated in \cite[Conjectures 3.1 \& 3.3]{hopkins} with the conjectures of Beilinson and Bloch. See the introduction of \cite{hopkins} for further details and \cite{Zemel} for related work on higher weight Gross--Kohnen--Zagier type theorems.
\end{remark}
 
We will need the following two lemmas for the proof of Theorem \ref{thm:CAJ}. Recall the correspondence $\Pi_{k,\varphi}:=\Pi_{k,k,\varphi}$ from $X_k$ to $W_{k,H}$ defined in \eqref{Pikk}.

\begin{lemma}\label{functoriality}
Under the assumptions of Theorem \ref{thm:CAJ},
let $\varphi : A \lra A'$ be an isogeny of degree $d_{\varphi}$ representing an element of $\Isog^{\cN}(A)$. Then, for all $f\in S_{k+2}(\Gamma_1(N))$, we have 
\[
\AJ_{W_k}((2^k k!)^2 d_{\varphi}^{r}\tilde{\Delta}_\varphi^{\HC})(\omega_f)=\AJ_{X_k}(\tilde{\Delta}_\varphi^{\GHC})(\Pi_{k, \varphi}^*(\omega_f)) \pmod {L_k}.
\]
\end{lemma} 

\begin{proof}
By functoriality of Abel--Jacobi maps with respect to correspondences \cite[Propositions 1,2 \& 4 iii)]{zucker}, the following diagram commutes:
\begin{equation*}
\begin{tikzcd}
\CH^{k+1}(X_k)_0(\C) \arrow{r}{\AJ_{X_k}^{k+1}} \arrow{d}[swap]{(\Pi_{k, \varphi})_*}
& J^{k+1}(X_k/\C) \arrow{d}{(\Pi_{k, \varphi}^*)^\vee} \\
\CH^{r+1}(W_k)_0(\C) \arrow{r}[swap]{\AJ_{W_k}^{r+1}} 
& J^{r+1}(W_k/\C).
\end{tikzcd}
\end{equation*}
By Proposition \ref{prop:relation} (with $k'=k$), we have $(\Pi_{k,\varphi})_* (\Delta_{\varphi}^{\GHC})=d_{\varphi}^r \Delta_{\varphi}^{\HC}$. 
Since $m_{k,k}=2^k k!m_{k,0}$, we deduce that $(\Pi_{k,\varphi})_* (\tilde{\Delta}_{\varphi}^{\GHC})=2^k k! d_{\varphi}^r \tilde{\Delta}_{\varphi}^{\HC}$. 
The result then follows by \eqref{eq:aj}.
\end{proof}

Let $\omega_A \in H^{1,0}(A_\C)$ be a non-zero differential form. Recall that the isomorphism $\End_H(A)\simeq \oh_K$ is chosen such that $[\alpha]^*\omega_A = \alpha\omega_A$ for all $\alpha \in \oh_K$. The choice of $\omega_A$ determines a generator $\eta_A$ of $H^{0,1}(A_\C)$ by the condition $\langle \omega_A, \eta_A\rangle=1$, where $\langle \: \: , \: \rangle$ denotes the cup-product on the de Rham cohomology of $A_\C$. The generator $\eta_A$ satisfies $[\alpha]^*\eta_A=\bar{\alpha}\eta_A$. For $0\leq j\leq k$, define 
\begin{equation}\label{omegaA}
\omega_A^{j}\eta_A^{k-j}:=\frac{j!(k-j)!}{k!}\sum_{\substack{I\subset \{ 1, \ldots, k \} \\ \vert I \vert=j}} \pr_1^*\varpi_{1, I}\wedge \ldots \wedge \pr_k^*\varpi_{k, I}, 
\end{equation}
where $\varpi_{i, I}$ is either $\omega_A$ or $\eta_A$ depending on whether $i\in I$ or $i\not\in I$. A basis of $\sym^k H^1_{\dR}(A_\C)$ is then given by $\{ \omega_A^{j}\eta_A^{k-j} \}$ for $0\leq j\leq k$. The cup product $\langle \: \: , \: \rangle$ induces a self-duality
\[
\langle \: \: , \: \rangle_{A^k} : \sym^k H^1_{\dR}(A_\C) \times \sym^k H^1_{\dR}(A_\C) \lra \C,
\]
given by 
\begin{equation}\label{duality}
\langle x_1 \cdots x_k, y_1 \cdots y_k\rangle_{A^k} := \frac{1}{k!} \sum_{\sigma\in \Sigma_k} \langle x_1, y_{\sigma(1)} \rangle\ldots  \langle x_k, y_{\sigma(k)} \rangle.
\end{equation}

\begin{lemma}\label{pullback}
Under the assumptions of Theorem \ref{thm:CAJ},
let $\varphi : A \lra A'$ be an isogeny of degree $d_{\varphi}$ representing an element of $\Isog^{\cN}(A)$. Let $\omega_A \in H^{1,0}(A_\C)$ be a non-zero differential form.
Then, for all $f\in S_{k+2}(\Gamma_1(N))$, we have 
$
\epsilon_{X_k}\Pi_{k, \varphi}^*(\omega_f)=(2d_\varphi \sqrt{-d_K})^r \omega_f \wedge \omega_A^r\eta_A^r.
$
\end{lemma}

\begin{remark}
The right hand side of the equality in Lemma \ref{pullback} does not depend on the choice of non-zero differential $\omega_A$ since scaling $\omega_A$ by $\lambda\in \C^\times$ leads to a scaling of $\eta_A$ by $\lambda^{-1}$.  
\end{remark}

\begin{proof}
The correspondence $\Pi_{k, \varphi}$ induces a pull-back map on de Rham cohomology 
\[
\Pi_{k, \varphi}^* : \Fil^{r+1}H_{\dR}^{k+1}(W_{k,\C}) \lra \Fil^{k+1}H_{\dR}^{2k+1}(X_{k, \C}),
\]
given by the usual formula $\Pi_{k, \varphi}^*(\omega)=(\pi_{01})_*(\cl_{\dR}(\Pi_{k, \varphi}) \wedge \pi_2^*(\omega))$, where $\cl_{\dR}(\Pi_{k, \varphi}) \in H_{\dR}^{3k+2}(Z_{k,\C})$ is the de Rham cycle class of $\Pi_{k, \varphi}$. We are only interested in the piece of $\Fil^{k+1}H_{\dR}^{2k+1}(X_{k, \C})$ that survives after applying $\epsilon_{X_k}$. Since $$\Fil^{k+1}H_{\dR}^{2k+1}(X_{k, \C})^\perp=\Fil^{k+1}H_{\dR}^{2k+1}(X_{k, \C}),$$ and $$H_{\dR}^{2k+1}(X_{k, \C})=\Fil^{k+1}H_{\dR}^{2k+1}(X_{k, \C})\oplus \overline{\Fil^{k+1}H_{\dR}^{2k+1}(X_{k,\C})},$$ 
we see that the dual of $\Fil^{k+1} \epsilon_{X_k} H_{\dR}^{2k+1}(X_{k, \C})$ with respect to the de Rham pairing is its complex conjugate. By Proposition \ref{eqn:filr-para-recall}, $\epsilon_{X_k}\Pi_{k, \varphi}^*(\omega_f)$ is completely determined by the values
\[
\langle \epsilon_{X_k}\Pi_{k, \varphi}^*(\omega_f), \overline{\omega}_g \wedge \omega_A^j\eta_A^{k-j} \rangle_{X_k}, \quad \text{ for all eigenforms } g\in S_{k+2}(\Gamma_1(N)) \text{ and } 0\leq j\leq k. 
\]
Using the notations of \eqref{diag1} and properties of the de Rham pairing, we compute that
\begin{align*}
\langle \epsilon_{X_k}\Pi_{k, \varphi}^*(\omega_f), \overline{\omega}_g \wedge \omega_A^j\eta_A^{k-j} \rangle_{X_k} & =
\langle (\pi_{01})_*(\cl_{\dR}(\Pi_{k, \varphi}) \wedge \pi_2^*(\omega_f)), \overline{\omega}_g \wedge \omega_A^j\eta_A^{k-j} \rangle_{X_k}\\
& = \langle \cl_{\dR}(\Pi_{k, \varphi}) \wedge \pi_2^*(\omega_f), \pi_{01}^*(\overline{\omega}_g \wedge \omega_A^j\eta_A^{k-j}) \rangle_{Z_k} \\
& =- \langle \cl_{\dR}(\Pi_{k, \varphi}), \pi_{0}^*(\overline{\omega}_g) \wedge \pi_1^*(\omega_A^j\eta_A^{k-j}) \wedge \pi_2^*(\omega_f) \rangle_{Z_k} \\
& =- \langle (\Psi_{k,\varphi})_*(\cl_{\dR}(W_k\times A^r)), \pi_{0}^*(\overline{\omega}_g) \wedge \pi_1^*(\omega_A^j\eta_A^{k-j}) \wedge \pi_2^*(\omega_f) \rangle_{Z_k} \\
& =- \langle \cl_{\dR}(W_k\times A^r), (\Psi_{k,\varphi})^*(\pi_{0}^*(\overline{\omega}_g) \wedge \pi_1^*(\omega_A^j\eta_A^{k-j}) \wedge \pi_2^*(\omega_f)) \rangle_{Z_k} \\
& = \langle \omega_f\wedge \overline{\omega}_g, ((\id_A, [d_\varphi \sqrt{-d_K}])^r)^*(\omega_A^j\eta_A^{k-j}) ) \rangle_{W_k\times A^r} \\
& = \langle\omega_f, \overline{\omega}_g\rangle_{W_k} \left( \int_{A^r} ((\id_A, [d_\varphi \sqrt{-d_K}])^r)^*(\omega_A^j\eta_A^{k-j}) \right).
\end{align*}
Observe that $((\id_A, [d_\varphi \sqrt{-d_K}])^r)^*(\pr_1^*\varpi_{1, I}\wedge \ldots \wedge \pr_k^*\varpi_{k, I})\neq 0$ if and only if $\varpi_{2l-1, I}\neq \varpi_{2l, I}$ for all $1\leq l\leq r$. In particular $j$ must equal $r$, and there are $2^r$ sets $I$ of length $r$ that satisfy this condition. For such a set $I$, observe for all $1\leq l\leq r$ that 
$$(\id_A, [d_\varphi \sqrt{-d_K}])^*(\pr_{2l-1}^*(\varpi_{2l-1})\wedge \pr_{2l}^*(\varpi_{2l, I}))=-d_{\varphi}\sqrt{-d_K}\pr_l^*(\omega_A \wedge \eta_A).$$ 
From the defining equation \eqref{omegaA}, we see that $((\id_A, [d_\varphi \sqrt{-d_K}])^r)^*(\omega_A^j\eta_A^{k-j})=0$ for $j\neq r$, and
\[
((\id_A, [d_\varphi \sqrt{-d_K}])^r)^*(\omega_A^r\eta_A^{r})=
(-d_{\varphi}\sqrt{-d_K})^r\frac{2^r(r!)^2}{(2r)!}\pr_1^*(\omega_A \wedge \eta_A)\wedge \ldots \wedge \pr_r^*(\omega_A \wedge \eta_A).
\]
This shows that for any eigenform $g$ and any $0\leq j\leq k$, we have 
 \begin{equation}\label{omegaf}
 \langle \epsilon_{X_k}\Pi_{k, \varphi}^*(\omega_f), \overline{\omega}_g \wedge \omega_A^j\eta_A^{k-j} \rangle_{X_k} = (-2d_\varphi\sqrt{-d_K})^r \frac{(r!)^2}{(2r)!} \langle \omega_f, \overline{\omega}_g\rangle_{W_k} \delta_{jr},
 \end{equation}
 where $\delta_{jr}=1$ if $j=r$ and $0$ otherwise.
 
 Observe from \eqref{duality} that 
 \[
  \langle \omega_f\wedge \omega_A^r\eta_A^r, \overline{\omega}_g \wedge \omega_A^j\eta_A^{k-j} \rangle_{X_k}=\langle \omega_f, \overline{\omega}_g\rangle_{W_k} \langle \omega_A^r\eta_A^r, \omega_A^j\eta_A^{k-j}\rangle_{A^k}=(-1)^r\frac{(r!)^2}{(2r)!} \langle \omega_f, \overline{\omega}_g\rangle_{W_k}\delta_{jr}. 
 \]
 The result follows by comparing with \eqref{omegaf}.
 \end{proof}
 
 \begin{proof}[Proof of Theorem \ref{thm:CAJ}]
Using Lemma \ref{functoriality} and the equality $\AJ_{X_k}=\tilde{\epsilon}_{X_k} \circ \AJ_{X_k}^{k+1}$, we see that 
 \[
 \AJ_{W_k}((2^k k!)^2 d_{\varphi}^{r}\tilde{\Delta}_{\varphi}^{\HC})(\omega_f)=\AJ^{k+1}_{X_k}(\tilde{\Delta}_\varphi^{\GHC})(\tilde{\epsilon}_{X_k}\Pi_{k, \varphi}^*(\omega_f)) \pmod {L_k}.
 \]
%by functoriality of Abel--Jacobi maps with respect to the correspondence $\tilde{\epsilon}_{X_k}$. 
Let $\omega_A\in H^{1,0}(A_\C)$ be the non-zero differential form given by $\omega_A=\varphi^*(2\pi i dw)$. By Lemma \ref{pullback}, we obtain 
 \begin{equation}\label{CAJHC}
\AJ_{W_k}((2^k k!)^2 d_{\varphi}^{r}\tilde{\Delta}_{\varphi}^{\HC})(\omega_f)=m_{k,k}(2d_{\varphi}\sqrt{-d_K})^r\AJ^{k+1}_{X_k}(\tilde{\Delta}_\varphi^{\GHC})(\omega_f\wedge \omega_A^r\eta_A^r) \pmod {L_k}.
 \end{equation}
The result follows by applying \cite[Theorem 1]{bdlp}, remembering that $\tilde{\Delta}_\varphi^{\GHC}=m_{k,k}\Delta_\varphi^{\GHC}$.
%(the explicit normalisation by $m_{k,k}$ that we use here is implicit therein by \cite[Remark 2]{bdlp}).
 \end{proof}

\section{Bloch's map on torsion cycles}\label{s:bloch}

In this section, we recall the existence and properties of an \'etale cycle class map defined on torsion cycles first considered by Bloch \cite{bloch1}. Its restriction to null-homologous cycles admits a comparison with the complex Abel--Jacobi map restricted to torsion cycles. As we will show, it follows that it factors through algebraic equivalence when composed with the correspondence $\tilde{\epsilon}_{W_k}$ of Definition \ref{rem:intcycles}. The resulting composition map plays a key role in the proof of Theorem \ref{intro:main:thm} in Section \ref{s:gr}. For a more complete account of Bloch's map, we refer the reader to \cite[\S 1.5.2]{mythesis}. 

\subsection{Basic properties}

Let $V$ denote a smooth projective variety of dimension $d$ defined over a number field $F$ and let $\ell$ denote a fixed prime. 
For all non-negative integers $n, j$ and $\nu$,
we use the convention 
$
H^n_{\et}(V_{\bar{F}}, \Z/\ell^\nu\Z(j)):=H_{\et}^n(V_{\bar{F}}, \mu_{\ell^\nu}^{\otimes j}),
$
where  $\mu_{\ell^\nu}$ is the \'etale sheaf of $\ell^\nu$-th roots of unity.
There are natural maps 
\begin{equation}\label{map:etlim}
H^n_{\et}(V_{\bar{F}}, \Z/\ell^{\nu}\Z(j))\lra H^n_{\et}(V_{\bar{F}}, \Z/\ell^{\nu+1}\Z(j))
\end{equation}
induced by the maps $\Z/\ell^{\nu}\Z\hookrightarrow \Z/\ell^{\nu+1}\Z$ sending $m\mapsto \ell m$ or by the inclusions $\mu_{\ell^{\nu}}\hookrightarrow \mu_{\ell^{\nu+1}}$. By taking the direct limit over $\nu$, we obtain the cohomology groups of $V$ with $\ell$-torsion coefficients
\begin{equation}\label{def:torscoho}
H^n_{\et}(V_{\bar{F}}, \Q_\ell/\Z_\ell(j)):= \lim_{\longrightarrow} H^n_{\et}(V_{\bar{F}}, \Z/\ell^{\nu}\Z(j)). 
\end{equation}
Viewing $\Q_\ell/\Z_\ell$ as a torsion \'etale sheaf on $V$, there is a natural isomorphism 
\begin{equation}
H^n_{\et}(V_{\bar{F}}, \Q_\ell/\Z_\ell)\otimes _{\Q_\ell/\Z_\ell} \Q_\ell/\Z_\ell(j) \simeq H^n_{\et}(V_{\bar{F}}, \Q_\ell/\Z_\ell(j))
\end{equation}
where the right hand side cohomology group is defined by \eqref{def:torscoho}.

Let $\CH^j(V_{\bar F})[\ell^\infty]$ denote the power-of-$\ell$ torsion subgroup of the Chow group. Bloch has defined in \cite{bloch1} a map 
\begin{equation}\label{blochmap}
\lambda_{V,\ell}^j : \CH^j(V_{\bar F})[\ell^\infty]\lra H^{2j-1}_{\et}(V_{\bar{F}}, \Q_\ell/\Z_\ell(j)).
\end{equation}
whose restriction to null-homologous cycles can be regarded as an arithmetic avatar of the complex Abel--Jacobi map on torsion (see Section \ref{s:blochcomp} below for a precise statement). 
The construction of the map is rather involved, and we therefore refer the reader to the original \cite[\S 2]{bloch1}, or alternatively to \cite[Appendix A]{ACMV}. 

%Next, we collect some of the properties of the Bloch map relevant for this paper. 
%Another brief overview of some of these properties is provided in \cite{schoen92}.

%Following \cite[\S 3]{bloch1}, we define the map induced by a correspondence on the target of the map $\lambda_{V,\ell}^j$.

%\begin{definition}\label{def:corr}
%Let $V$ and $W$ be two smooth projective varieties over $F$ with $d=\dim(V)$. Let $i$ and $j$ be two integers. If $T$ is a cycle on $V\times W$ of codimension $d+i-j$, then it induces a map $$T_* : H^{2j-1}_{\et}(V_{\bar{F}}, \Q_\ell/\Z_\ell(j))\lra H^{2i-1}_{\et}(W_{\bar{F}}, \Q_\ell/\Z_\ell(i)), \qquad T_*:= (\pr_W)_* \circ  (\cup [T]) \circ (\pr_V)^*,$$ 
%where $[T]\in H_{\et}^{2d+2i-2j}((V\times W)_{\bar F}, \Z_\ell(d+i-j))$ is the $\ell$-adic cycle class of $T$, $\pr_V$ and $\pr_W$ are the natural projections maps on $V\times W$, and the cup product is the one induced on \'etale cohomology by the natural bilinear map 
%$\Q_\ell/\Z_\ell(j)\times \Z_\ell(d+i-j)\lra \Q_\ell/\Z_\ell(d+i).$
%\end{definition}

\begin{proposition}\label{prop:bloch}
The Bloch map (\ref{blochmap}) is functorial with respect to correspondences and $\Gal(\bar{F}/F)$-equivariant.
\end{proposition}

\begin{proof}
Functoriality for correspondences is \cite[Proposition 3.5]{bloch1}. The Galois equivariance is \cite[Proposition A.22]{ACMV}. 
%; if $Z\in \CH^j(X)$ and $\Gamma\in \CH^r(X)(\ell)$, then 
%\[
%\lambda_\ell^{r+j}(\Gamma \cdot Z)=\lambda_\ell^r(\Gamma) \cup \cl_{\ell}(Z)
%\]
%where $\cup$ denotes the cup product on \'etale cohomology induced by the bilinear map 
%\[ \Q_\ell/\Z_\ell(r)\times \Z_\ell(j)\lra \Q_\ell/\Z_\ell(r+j). \]
\end{proof}

\subsection{Comparison with the complex Abel--Jacobi map}\label{s:blochcomp}

Using the description \eqref{PDinterjac} of the intermediate Jacobian along with the natural isomorphism of $\R$-vector spaces
\begin{equation}\label{iso:R}
H^{2j-1}(V_{\C}(\C), \R)\simeq H^{2j-1}(V_{\C}(\C), \C)/\Fil^{j} H_{\dR}^{2j-1}(V_{\C}),
\end{equation}
there is an identification of real tori
\begin{equation*}
J^j(V_\C)\simeq H^{2j-1}(V_{\C}(\C), \R)/ H^{2r-1}(V_{\C}(\C), \Z),
\end{equation*}
and thus an identification
\begin{equation}\label{jactors}  
J^{j}(V_{\C})_{\tors} \simeq H^{2j-1}(V_{\C}(\C), \Q)/ H^{2j-1}(V_{\C}(\C), \Z).
\end{equation}
From the long exact sequence in singular cohomology associated to the short exact sequence 
\begin{equation}
0\lra \Z\lra\Q\lra\Q/\Z\lra 0
\end{equation}
we deduce a short exact sequence
\begin{equation}\label{map:u}
0\lra J^{j}(V_{\C})_{\tors}\overset{u}{\lra} H^{2j-1}(V_{\C}(\C), \Q/\Z)\lra H^{2j}(V_{\C}(\C), \Z)_{\tors}\lra 0,
\end{equation}
thus identifying $J^{j}(V_{\C})_{\tors}$ up to a finite group with $H^{2j-1}(V_{\C}(\C), \Q/\Z)$.

Composing the complex Abel--Jacobi map \eqref{map:CAJ} restricted to torsion with $u$ yields a map
\begin{equation}\label{AJcomplex}  
u\circ \AJ_V^j : \CH^{j}(V_{\C})_0[\ell^\infty]\lra H^{2j-1}(V_{\C}(\C), \Q_\ell/\Z_\ell).
\end{equation}

For each natural number $\nu$, there is a sequence of isomorphisms
\begin{equation}\label{iso:mu}
H^{2j-1}_{\et}(V_{\bar{F}}, \mu_{\ell^\nu}^{\otimes j})\simeq  H^{2j-1}_{\et}(V_{\C}, \mu_{\ell^\nu}^{\otimes j})\simeq H^{2j-1}(V_\C(\C), \mu_{\ell^\nu}^{\otimes j}).
\end{equation}
For the first isomorphism, apply \cite[VI Corollary 4.3]{milne} with respect to the complex embedding $\bar{F}\hookrightarrow \C$ fixed in Section \ref{s:convention}. The second isomorphism is an application of \cite[III Theorem 3.12]{milne}. Taking direct limits over $\nu$, we obtain a sequence of isomorphisms
\begin{equation}\label{map:comp}
\comp : H^{2j-1}_{\et}(V_{\bar{F}}, \Q_\ell/\Z_\ell (j)) \simeq H^{2j-1}_{\et}(V_{\C}, \Q_\ell/\Z_\ell (j)) 
 \simeq H^{2j-1}(V_{\C}(\C), \Q_\ell/\Z_\ell (j)).
\end{equation}

\begin{proposition}\label{prop:blochcomp}
If we identify $\Q_\ell/\Z_\ell\simeq \Q_\ell/\Z_\ell(j)$ by taking $e^{\frac{2\pi i}{\ell^\nu}}$ as the generator of the $\ell^\nu$-th roots of unity, then the diagram 
\begin{equation}\label{diag:blochcomp}
\begin{tikzcd}
\CH^j(V_{\bar F})_0[\ell^\infty] \arrow{r}{\lambda_\ell^r} \arrow{d}
& H_{\et}^{2j-1}(V_{\bar{F}}, \Q_\ell/\Z_\ell(j)) \arrow{d}{\comp}[swap]{\wr} \\
\CH^j(V_{\C})_0[\ell^\infty] \arrow{r}{u\circ \AJ^j_V} 
& H^{2j-1}(V_{\C}(\C), \Q_\ell/\Z_\ell)
\end{tikzcd}
\end{equation}
commutes, where the vertical left map is induced by the fixed embedding $\bar F \hookrightarrow \C$. 
\end{proposition}

\begin{proof}
This is \cite[Proposition 3.7]{bloch1}.
\end{proof}

\subsection{Factorisation through algebraic equivalence}

Let $k=2r\geq 2$ be an even integer. We now specialise to the case where $V$ is the Kuga--Sato variety $W_k$ of level $\Gamma_1(N)$  defined over $\Q$.
Recall the idempotent correspondence $\epsilon_{W_k}$ defined in \eqref{def:eps} and its normalisation $\tilde{\epsilon}_{W_k}$ defined in Definition \ref{rem:intcycles}, along with the fact that the composition of the complex Abel--Jacobi map with this correspondence can be viewed as a map \eqref{epsAJ}
\[
\AJ_{W_k}=\tilde{\epsilon}_{W_k}\circ \AJ_{W_k}^{r+1}: \CH^{r+1}(W_{k,\C})_0 \lra \frac{S_{k+2}(\Gamma_1(N))^\vee}{L_k}.
\]
%The following result will enable us to use Theorem \ref{thm:CAJ} to deduce information about the algebraic equivalence classes of Heegner cycles:
\begin{proposition}\label{lem:gr}
The map $\AJ_{W_k}$ factors through algebraic equivalence, giving rise to a map
\[
\AJ_{W_k} : \Gr^{r+1}(W_{k,\C})\lra \frac{S_{k+2}(\Gamma_1(N))^\vee}{L_k}.
\]
\end{proposition}

\begin{proof}
Let $\CH^{r+1}(W_{k,\C})_{\alg}$ denote the subgroup of $\CH^{r+1}(W_{k,\C})_0$ consisting of algebraically trivial cycles. The image of $\CH^{r+1}(W_{k,\C})_{\alg}$ under $\AJ_{W_k}^{r+1}$ lies in an abelian subvariety $J^{r+1}(W_{k,\C})_{\alg}$ of $J^{r+1}(W_{k,\C})$ whose cotangent space is contained in $H^{r+1, r}(W_{k,\C})$ \cite[\S 12.2.2]{voisin}. More precisely, if $T_\Z$ denotes the largest integral sub-Hodge structure of $H^{k+1}(W_{k,\C}(\C),\Z)$ of type $(r+1, r)+(r, r+1)$, then 
\[
\AJ_{W_k}^{r+1}(\CH^{r+1}(W_{k,\C})_{\alg})\subset J^{r+1}(W_{k,\C})_{\alg}=J(T_\Z):=T_\C/(\Fil^{r+1} T_\C \oplus T_\Z)\subset J^{r+1}(W_{k,\C}). 
\]
Recall that $\AJ_{W_k}=\tilde{\epsilon}_{W_k}\circ \AJ_{W_k}^{r+1}$ by definition.
Thus, in order to prove the proposition, it suffices to show that the map on complex tori $\tilde{\epsilon}_{W_k} : J^{r+1}(W_{k,\C})\lra J^{r+1}(W_{k,\C})$ restricts to the zero map on $J^{r+1}(W_{k,\C})_{\alg}$. This restriction is completely determined by the restriction of the map of Hodge structures $\tilde{\epsilon}_{W_k} : H^{k+1}(W_{k,\C}(\C),\Z)\lra H^{k+1}(W_{k,\C}(\C),\Z)$ to $T_\Z$.
By Proposition \ref{prop:epsiloncusp}, the motive $(W_k, \epsilon_{W_k}, 0)$ with rational coefficients is of pure Hodge type $(k+1, 0)+(0, k+1)$, and in particular 
\[
\epsilon_{W_k}(H^{r+1, r}(W_k)\oplus H^{r, r+1}(W_k))=0.
\]
It follows that $\tilde{\epsilon}_{W_k}(T_\C)=0$ and thus $\tilde{\epsilon}_{W_k}(T_\Z)=0$.
%It follows, as in the proof of \cite[Proposition 9]{bdlp}, that $\AJ_{W_k}^{r+1}(\CH^{r+1}(W_k)_{\alg})$ is annihilated by the projector $\epsilon_{W_k}$.
\end{proof}

Taking the direct sum of the $\ell$-adic Bloch maps \eqref{blochmap} over all primes $\ell$ yields a map
\begin{equation}\label{map:bloch}
\lambda_{W_k}^{r+1} : \CH^{r+1}(W_{k,\bar \Q})_{\tors} \lra H^{k+1}_{\et}(W_{k,\bar \Q}, \Q/\Z(r+1)),
\end{equation}
which is functorial with respect to correspondences by Proposition \ref{prop:bloch}. In view of evaluating this map on Heegner cycles, it thus suffices to consider its composition with $\tilde{\epsilon}_{W_k}$. Restricting to null-homologous cycles and composing with $\tilde{\epsilon}_{W_k}$ yields a map 
\begin{equation}\label{lambda0}
\lambda_{W_k}^{\circ} : \CH^{r+1}(W_{k,\bar \Q})_{0, \tors} \lra  \tilde{\epsilon}_{W_k} H^{k+1}_{\et}(W_{k,\bar \Q}, \Q/\Z(r+1))\subset H^{k+1}_{\et}(W_{k,\bar \Q}, \Q/\Z(r+1)).
\end{equation}
 
 \begin{proposition}\label{lem:grbloch}
The map \eqref{lambda0} factors through algebraic equivalence, giving rise to a map
\[
\lambda^{\circ}_{W_k} : \Gr^{r+1}(W_{k,\bar \Q})\lra  H^{k+1}_{\et}(W_{k,\bar \Q}, \Q/\Z(r+1)).
\]
\end{proposition}

\begin{proof}
The group $\CH^{r+1}(W_{k,\bar \Q})_{\alg}$ is divisible since (by definition of algebraic equivalence) it is generated by images under correspondences of $\bar{\Q}$-valued points on Jacobians of curves. It follows that there is an exact sequence of torsion subgroups
\begin{equation}\label{seq:div}
0\lra \CH^{r+1}(W_{k,\bar \Q})_{\alg, \tors}\lra \CH^{r+1}(W_{k,\bar \Q})_{0, \tors}\lra \Gr^{r+1}(W_{k,\bar{\Q}})_{\tors} \lra 0.
\end{equation}
In order to prove the result it thus suffices to show that the subgroup $\CH^{r+1}(W_{k,\bar \Q})_{\alg, \tors}$ lies in the kernel of \eqref{lambda0}.
By Proposition \ref{prop:blochcomp}, we have
\begin{equation}
\lambda_{W_k}^{\circ} = \tilde{\epsilon}_{W_k}\circ \comp^{-1} \circ u\circ \AJ_{W_k}^{r+1}.
\end{equation}
By compatibility of the comparison isomorphism (\ref{map:comp}) with correspondences (which follows from the compatibility of the cycle class maps with respect to the comparison isomorphism \cite[\S 5.3]{jannsen}), \begin{equation}
\lambda_{W_k}^{\circ} = \comp^{-1}\circ \tilde{\epsilon}_{W_k} \circ u\circ \AJ_{W_k}^{r+1}.
\end{equation}
From the natural compatibility of the map $u$ with correspondences, it follows that
\begin{equation}\label{eq:commu} 
\lambda_{W_k}^{\circ} = \comp^{-1}\circ u\circ \tilde{\epsilon}_{W_k} \circ \AJ_{W_k}^{r+1}=\comp^{-1}\circ u\circ \AJ_{W_k}.
\end{equation}
We have $\AJ_{W_k}(\CH^{r+1}(W_{k,\bar \Q})_{\alg, \tors})=0$ by Proposition \ref{lem:gr}.
\end{proof}

\section{A finiteness result for \'etale cohomology with torsion coefficients}\label{s:finiteness}

Let $k=2r\geq 2$ and $N\geq 5$ be integers. Let $W_k$ be the Kuga--Sato variety of level $\Gamma_1(N)$ over $\spec \Z[1/N]$ (constructed in \cite[Appendix]{bdp1}). Let $K$ be an imaginary quadratic field of discriminant $-d_K$ coprime to $N$ which satisfies the Heegner hypothesis with respect to $N$. Let $\cN$ be a choice of cyclic $N$-ideal of $K$. For each positive integer $n$, recall that $H_n$ denotes the ring class field of conductor $n$ over $K$, while $K_{\cN}$ is the ray class field of conductor $\cN$ over $K$. Let $H_\infty$ denote the compositum of the ring class fields $H_n$ for all square-free integers $n$ coprime to $N$. Let $F_n := K_{\cN}\cdot H_n$ and $F_\infty:=K_{\cN}\cdot H_\infty$. The goal of this section is to prove the following: 

\begin{proposition}\label{prop:finite}
With the above notations, the group $H^{k+1}_{\et}(W_{k,\bar \Q}, \Q/\Z(r+1))^{\Gal(\bar \Q/F_\infty)}$ is finite. 
\end{proposition} 

Before proving Proposition \ref{prop:finite}, we collect a preliminary result concerning the splitting behaviour of primes in the extension $F_\infty$ of $K$:

\begin{lemma}\label{lem:resdeg}
With the above notations, let $q$ be a prime which is coprime to $2N$ and inert in $K$. Let $\mathfrak{q}$ denote a prime of $H$ above $q$ and denote by $s$ its residual degree in the extension $K_{\cN}/H$. Then, for any square-free positive integer $n$ coprime to $N$, the residual degree of $\mathfrak{q}$ in the extension $F_n/H$ is equal to $s$.
\end{lemma}

\begin{proof}
This is \cite[Corollary 1.2]{mythesis}. The proof uses the fact that if $n$ is a square-free positive integer and $q$ is a rational prime which is inert in $K$, then the residual degree of $q\oh_K$ in the extension $H_n/K$ is equal to $1$ (see for instance \cite[Proposition 1.8]{mythesis}).
\end{proof}

\begin{proof}[Proof of Proposition \ref{prop:finite}]
Fix $q_1$ and $q_2$ two distinct primes which are coprime to $2N$ and inert in $K$. Let $i\in \{ 1,2\}$. The variety $W_k$ has good reduction at $q_i$, and we may consider the reduction $W_{k,\F_{q_i}}$ over $\F_{q_i}$. The embeddings fixed in Section \ref{s:convention} determine primes $\mathfrak{q}_i$ in $H$ and $\mathfrak{q}_i^\infty$ in $F_\infty$ above $q_i\oh_K$. If $s_i$ denotes the residual degree of $\mathfrak{q}_i$ in $K_{\cN}/H$, then the residual degree of $\mathfrak{q}_i$ in $F_\infty/H$ is $s_i$ by Lemma \ref{lem:resdeg}. Since $q$ is inert in $K$, $q\oh_K$ splits completely in $H$. It follows that the residual degree of $q\oh_K$ in $H/K$ is equal to $1$. By multiplicativity of residual degrees in extension towers, we conclude that $q_i$ has residual degree $r_i:=2s_i$ in the extension $F_\infty/\Q$.
Let $D_i$ denote the decomposition group of $\Gal(\bar \Q/F_\infty)$ of a prime above $\mathfrak{q}_i^\infty$. 

Let $\ell$ be a prime and choose $i\in\{ 1,2\}$ such that $\ell\neq q_i$. Note that we may choose $i=1$ except when $\ell=q_1$ in which case we must choose $i=2$.
Using \cite[VI Corollary 4.2]{milne} and taking direct limits, we obtain an isomorphism
\[
H^{k+1}_{\et}(W_{k,\bar \Q}, \Q_\ell/\Z_\ell(r+1))^{D_i} \simeq H^{k+1}_{\et}(W_{k,\bar{\F}_{q_i}}, \Q_\ell/\Z_\ell(r+1))^{\Gal(\bar{\F}_{q_i}/\F_{q_i^{r_i}})}.
\]
In particular, $H^{k+1}_{\et}(W_{k,\bar \Q}, \Q/\Z(r+1))^{\Gal(\bar \Q/F_\infty)}$ injects into
\begin{equation}\label{directsum}
H^{k+1}_{\et}(W_{k,\bar{\F}_{q_2}}, \Q_{q_1}/\Z_{q_1}(r+1))^{\Gal(\bar{\F}_{q_2}/\F_{q_2^{r_2}})} \oplus \bigoplus_{\ell\neq q_1} H^{k+1}_{\et}(W_{k,\bar{\F}_{q_1}}, \Q_\ell/\Z_\ell(r+1))^{\Gal(\bar{\F}_{q_1}/\F_{q_1^{r_1}})}.
\end{equation}

We have reduced the proof to showing that the group \eqref{directsum} is finite. Let $\ell$ be a prime and choose $i\in \{ 1,2 \}$ such that $\ell\neq q_i$. From the short exact sequence $0\lra \Z_\ell\lra \Q_\ell\lra \Q_\ell/\Z_\ell\lra 0$, we deduce a short exact sequence 
\[
0\lra \frac{H^{k+1}_{\et}(W_{k,\bar{\F}_{q_i}}, \Q_\ell(r+1))}{H^{k+1}_{\et}(W_{k,\bar{\F}_{q_i}}, \Z_\ell(r+1))} \lra H^{k+1}_{\et}(W_{k,\bar{\F}_{q_i}}, \Q_\ell/\Z_\ell(r+1)) \lra H^{k+2}_{\et}(W_{k,\bar{\F}_{q_i}}, \Z_\ell(r+1))_{\tors}
\lra 0.
\]

The group on the right hand side is finite and trivial for all but finitely many $\ell$. Indeed, using \cite[VI Corollary 4.2 \& 4.3]{milne} and taking inverse limits gives an isomorphism 
\[
H^{k+2}_{\et}(W_{k,\bar{\F}_{q_i}}, \Z_\ell(r+1))\simeq H^{k+2}_{\et}(W_{k,\bar{\Q}}, \Z_\ell(r+1)),
\]
which in turn is isomorphic to $H^{k+2}_{\et}(W_{k,\C}(\C), \Z)(r+1)\otimes \Z_\ell$ by the comparison isomorphism \cite[III Theorem 3.12]{milne}. The claim follows since $H^{k+2}_{\et}(W_{k,\C}(\C), \Z)_{\tors}$ is finite. 

The number of fixed points under the action of $\Gal(\bar{F}_{q_i}/\F_{q_i^{r_i}})$ of the left hand side of the above short exact sequence is equal to 
$
\vert \det(1-\Frob_{q_i^{r_i}} \vert H^{k+1}_{\et}(W_{k,\bar{\F}_{q_i}}, \Q_\ell(r+1))) \vert,
$
which is finite and independent of the prime $\ell$ by Deligne's theorem (the Weil conjecture) \cite{deligneW}.

In conclusion, each term in the infinite direct sum \eqref{directsum} is finite, and trivial for all but finitely many terms.
\end{proof}

\begin{definition}\label{def:Mr}
Let $M_r:=\vert H^{k+1}_{\et}(W_{k,\bar \Q}, \Q/\Z(r+1))^{\Gal(\bar \Q/F_\infty)} \vert$, which is finite by Proposition \ref{prop:finite}.
\end{definition}

\begin{corollary}\label{coro:finite}
With the above notations, if $n$ is a square-free integer coprime to $N$, then the image of $\Gr^{r+1}(W_{k,F_n})_{\tors}$ under the map $\lambda_{W_k}^\circ$ of Proposition \ref{lem:grbloch} is annihilated by $M_r$.
\end{corollary}

\begin{proof}
By the Galois equivariance of the Bloch map (Proposition \ref{prop:bloch}), the image of $\Gr^{r+1}(W_{k,F_n})_{\tors}$ under $\lambda_{W_k}^\circ$ lies in $H^{k+1}_{\et}(W_{k,\bar \Q}, \Q/\Z(r+1))^{\Gal(\bar \Q/F_n)}$, a subgroup of $H^{k+1}_{\et}(W_{k,\bar \Q}, \Q/\Z(r+1))^{\Gal(\bar \Q/F_\infty)}$ by definition of $F_\infty$. The result follows by definition of $M_r$ (Definition \ref{def:Mr}).
\end{proof}

 \section{Explicit isogenies}\label{s:explicit}

In the rest of this paper, we will focus on a particular subcollection of Heegner cycles and their properties. These are indexed by certain explicit (isomorphism classes of) isogenies. In Section \ref{s:gr}, will prove that this subcollection generates a subgroup of infinite rank modulo algebraic equivalence. 

Fix an imaginary quadratic field $K$ with ring of integers $\cO_K$ and discriminant $-d_K$ coprime to $N$. Assume that $K$ satisfies the Heegner hypothesis with respect to $N$, and let $\cN$ denote a choice of cyclic $N$-ideal of $\oh_K$.
Let $A$ be an elliptic curve with CM by $\oh_K$ over the Hilbert class field $H$ of $K$. Choose the complex embedding $H\hookrightarrow \C$ of Section \ref{s:convention} such that $A_\C=\C/\oh_K$.
Let $\tau:=(-d_K+\sqrt{-d_K})/2\in \cH$ denote the standard generator of $\oh_K$, so that $\oh_K=\langle 1, \tau \rangle:=\Z\oplus \Z\tau$. It satisfies the quadratic equation $\tau^2+d_K\tau+d_K(d_K+1)/4=0$. Note that the coefficient $d_K(d_K+1)/4$ is integral since $-d_K\equiv 0,1 \pmod 4$.

\subsection{Explicit $q$-isogenies}\label{s:qisog}

Let $q$ be an odd prime which is coprime to $d_K$. Consider the $q+1$ lattices $\Lambda_{q,\beta}:=\langle 1,\tau_{q,\beta} \rangle=\Z\oplus \Z\tau_{q,\beta}$ in $\C$ indexed by $\beta\in \PP^1(\F_q)$, where 
\[
\tau_{q,\beta}:=
\begin{cases}
q\tau & \text{ if } \beta=\infty \\
\frac{\tau+\beta}{q} &  \text{ if } \beta\neq \infty.
\end{cases}
\]
Observe that $\Lambda_{q,\infty}$ is the order $\oh_{q}$ of $K$ of conductor $q$. 

There are natural isogenies $\varphi_{q,\beta} : \C/\oh_K \lra \C/\Lambda_{q,\beta}$ of complex tori defined as follows:
\begin{itemize}
 \item If $\beta=\infty$, then $\varphi_{q,\infty}$ is the natural quotient map induced by the inclusion of lattices $\langle 1, \tau \rangle\subset \langle 1/q, \tau \rangle$ composed with multiplication by $q$: 
\[
\varphi_{q,\infty} : \C/\langle 1,\tau\rangle \overset{\quot}{\lra} \C/\langle 1/q, \tau \rangle \overset{[q]}{\lra} \C/\langle 1, \tau_{q,\infty} \rangle.
\] 
\item If $\beta \neq \infty$, then $\varphi_{q,\beta}$ is given by the quotient map
\[
\varphi_{q,\beta} : \C/\langle 1, \tau\rangle \overset{=}{\lra}  \C/\langle 1, \tau+\beta \rangle \overset{\quot}{\lra}  \C/\langle 1, (\tau+\beta)/q \rangle,
\]
induced by the inclusion of lattices $\langle 1, \tau+\beta \rangle\subset \langle 1, (\tau+\beta)/q \rangle$. 
\end{itemize}
Observe that the degree $d_{\varphi_{q,\beta}}$ of $\varphi_{q,\beta}$ is equal to $q$ for all $\beta\in \PP^1(\F_q)$. 
%$$ \Lambda_{p, q, \infty} :=\Z\frac{1}{pq}\oplus \Z\tau, \qquad  \Lambda_{p, q, \beta} :=\Z\frac{1}{p}\oplus \Z\frac{\tau+\beta}{q}, \ \ \mbox{ for } 0\le \beta \le q-1,$$
%which each contain $\cO_K$ with index $pq$, and let $A_{p,q,\beta} $ be the  elliptic curve whose complex points are isomorphic to $\C/\Lambda_{p,q,\beta}$.

\begin{proposition}\label{prop:inert}
If $q$ is an odd prime which is inert in $K$, then $\Lambda_{q,\beta}$ is a proper fractional $\oh_{q}$-ideal for all $\beta\in \PP^1(\F_q)$.
\end{proposition}

\begin{proof}
The statement is clear for $\beta=\infty$, hence we assume $\beta\in \F_q$. The element $\tau_{q,\beta}\in \cH$ satisfies the quadratic equation $q^2\tau_{q,\beta}^2+q(d_K-2\beta)\tau_{q,\beta}+(\beta^2-d_K\beta+d_K(d_K+1)/4)=0$. If $\beta^2-d_K\beta+d_K(d_K+1)/4\equiv 0 \pmod q$, then $-d_K\equiv (2\beta-d_K)^2 \pmod q$, hence $\left( \frac{-d_K}{q} \right)=1$ and $q$ splits in $K$. Since $q$ is assumed to be inert, we conclude that $q$ does not divide the constant coefficient. Thus, the coefficients of the above quadratic equation satisfied by $\tau_{q,\beta}$ have no common factors. By \cite[Lemma 7.5]{cox}, $\Lambda_{q,\beta}$ is a proper fractional $\langle 1, q^2\tau_{q,\beta}\rangle$-ideal. The result follows by observing that $\langle 1, q^2\tau_{q,\beta}\rangle=\langle 1, q(\tau+\beta)\rangle=\oh_q$.
\end{proof}

\begin{proposition}\label{prop:descend}
Let $q$ be an odd prime which is inert in $K$ and coprime to $N$. For all $\beta\in \PP^1(\F_q)$, there exists an elliptic curve $A_{q,\beta}$ with CM by $\oh_{q}$ defined, along with its complex multiplication, over the ring class field $H_{q}$ such that $A_{q,\beta,\C}=\C/\Lambda_{q,\beta}$ and the isogeny of complex tori $\varphi_{q,\beta}$ descends to an isogeny $\varphi_{q,\beta} : A\lra A_{q,\beta}$ giving rise to an isomorphism class $(\varphi_{q,\beta}, A_{q,\beta})\in \Isog^{\cN}(A)$ with field of definition $H_{q}$.
\end{proposition}

\begin{proof}
By Proposition \ref{prop:inert}, the elliptic curve $\C/\Lambda_{q,\beta}$ has CM by $\oh_q$. The proposition is then a consequence of the main theorem of complex multiplication \cite[Theorem 11.1]{cox}.
Note that the assumption that $q$ is coprime to $N$ guarantees that the classes $(\varphi_{q,\beta}, A_{p,q,\beta})$ belong to $\Isog^{\cN}(A)$. 
\end{proof}

\begin{lemma}\label{lem:qinert}
Let $q$ be a prime which is inert in $K$, and let $u_K:=\vert \oh_K^\times \vert/2$. Then the extension $H_q/H$ is cyclic of order $(q+1)/u_K$.
\end{lemma}

\begin{remark}
Note that $u_K=1 \text{ if } d_K\neq 3,4$.
\end{remark}

\begin{proof}
Artin reciprocity yields an isomorphism
\begin{equation}\label{iso:gal}
(\oh_K/q\oh_K)^\times/\oh_K^\times(\Z/q\Z)^\times\simeq \Gal(H_q/H),
\end{equation}
by mapping $c$ to the Artin symbol $[c\oh_K, H_q/H]$ (see \cite[Eq. (7.27)]{cox}). Since $q$ is assumed to be inert, we have $(\oh_K/q\oh_K)^\times=\F_{q^2}^\times$ and the result follows.
\end{proof}

\begin{proposition}\label{prop:trans}
Let $q$ be an odd prime which is inert in $K$ and coprime to $N$. The action of the Galois group $\Gal(H_q/H)$ on the subset $\{ (\varphi_{q,\beta}, A_{q,\beta}) \: \vert \: \beta\in \PP^1(\F_q) \}\subset \Isog^{\cN}(A)$ is simply transitive. 
\end{proposition}

\begin{proof}
As in Section \ref{s:isog}, $\Gal(\bar H/H)$ naturally acts on $\Isog^{\cN}(A)$ since $A$ is defined over $H$. There are $q+1$ isogenies from $A$ of degree $q$, namely the isogenies $\varphi_{q,\beta}$ for $\beta\in \PP^1(\F_q)$. Any isogeny $\varphi : A\lra A'$ is completely determined by its kernel $A[\varphi] \subset A(\bar H)$. Two isomorphism classes $(\varphi_1, A_1)$ and $(\varphi_2, A_2)$ are equal if and only if there exists $\psi\in \Aut(A)/\langle \pm 1\rangle$ such that $\psi(A[\varphi_1])=A[\varphi_2]$ (the effect of the automorphism $-1$ being trivial). Since $A$ has CM  by $\oh_K$, we have $\Aut(A)=\oh_K^\times$. The set $\{ (\varphi_{q,\beta}, A_{q,\beta}) \: \vert \: \beta\in \PP^1(\F_q) \}$ of isomorphism classes of isogenies of degree $q$ therefore has order $(q+1)/u_K$. By Proposition \ref{prop:descend}, since $q$ is inert we have $\{ (\varphi_{q,\beta}, A_{q,\beta}) \: \vert \: \beta\in \PP^1(\F_q) \} \subset \Isog_q^{\cN}(A)$ (in the notations of Section \ref{s:isog}). It follows that the action of $\Gal(H_q/H)$ on $\{ (\varphi_{q,\beta}, A_{q,\beta}) \: \vert \: \beta\in \PP^1(\F_q) \}$ is simple. The transitivity then follows from the fact that  $\{ (\varphi_{q,\beta}, A_{q,\beta}) \: \vert \: \beta\in \PP^1(\F_q) \} $ and $\Gal(H_q/H)$ have the same order by Lemma \ref{lem:qinert}.
\end{proof}

\begin{remark}\label{rem:qsplit}
In contrast, if $q$ is an odd prime that splits in $K$, then the proof of Proposition \ref{prop:inert} shows that for exactly two choices of $\beta\in \F_q$, say $\beta_1$ and $\beta_2$, the constant term $\beta^2-d_K\beta+d_K(d_K+1)/4$ of the quadratic equation satisfied by $\tau_{q,\beta}$ is divisible by $q$. Hence, for $i\in \{1,2\}$, $\tau_{q,\beta_i}$ satisfies the equation $q\tau_{q,\beta_i}^2+(d_K-2\beta_i)\tau_{q,\beta_i}+(\beta_i^2-d_K\beta_i+d_K(d_K+1)/4)/q=0$ with coprime coefficients. By \cite[Lemma 7.5]{cox}, $\Lambda_{q,\beta_i}$ is a proper fractional $\langle 1, q\tau_{q,\beta_i}\rangle$-module for $i\in \{1,2\}$. But $\langle 1, q\tau_{q,\beta}\rangle=\oh_K$ for any $\beta\in \F_q$, hence the elliptic curves $\C/\Lambda_{q,\beta_1}$ and $\C/\Lambda_{q,\beta_2}$ have CM by $\oh_K$ and can be defined over $H$. 
Alternatively, 
writing $q=\mathfrak{q}\bar{\mathfrak{q}}$ for some prime ideal $\mathfrak{q}$ of $K$, these two elliptic curves along with their cyclic $q$-isogenies can be described as $\varphi_{\mathfrak{q}} : A\lra A/A[\mathfrak{q}]$ and $\varphi_{\bar{\mathfrak{q}}} : A\lra A/A[\bar{\mathfrak{q}}]$. In conclusion, there are $q-1$ cyclic $q$-isogenies with CM by $\oh_q$. Their isogeny classes form a set of order $(q-1)/u_K$ given by a single orbit under the action of $\Gal(H_q/H)$, which in the split case has order $(q-1)/u_K$ by \eqref{iso:gal}. Working under the assumption that $q$ is inert is not strictly speaking necessary for our method, but it does simplify the notations and arguments a little bit.
\end{remark}

\subsection{Explicit $\Gamma_1(N)$-level structure}\label{s:levelstruc}

Recall that $A_\C=\C/\oh_K$ is an elliptic curve with CM by $\oh_K$ over $H$, $\oh_K=\langle 1,\tau\rangle$ with $\tau:=(-d_K+\sqrt{-d_K})/2$ satisfying $\tau^2+d_K\tau+d_K(d_K+1)/4=0$, and $\cN$ is a cyclic $N$-ideal of $\oh_K$.
%Let $k=2r\geq 2$ and $N\geq 5$ be integers. Fix an imaginary quadratic field $K$ with ring of integers $\cO_K$ and discriminant $-d_K$ coprime to $N$. Assume that $K$ satisfies the Heegner hypothesis with respect to $N$, and let $\cN$ denote a choice of cyclic $N$-ideal of $\oh_K$.
%Let $H$ be the Hilbert class field of $K$. Let $A$ be an elliptic curve with CM by $\oh_K$ over $H$.  
%Let $\tau:=(-d_K+\sqrt{-d_K})/2$ denote the standard generator of $\oh_K$ as in the previous section, so that $\oh_K=\langle 1,\tau\rangle$. Recall that it satisfies the quadratic equation $\tau^2+d_K\tau+d_K(d_K+1)/4=0$. Choose the embedding $H\hookrightarrow \C$ of Section \ref{s:convention} such that $A_{\C}=\C/\oh_K$. 

%\begin{iassumption}\label{Nprime}
%The integer $N$ is prime.
%\end{iassumption}

\begin{proposition}\label{prop:t}
With the above notations, a generator $t$ of the cyclic group $A[\cN]$ must be of the form $t=(c\tau+d)/N + \langle 1, \tau \rangle$ for some integers $c, d \in \Z$ with $\gcd(c,d,N)=1$ and $c \not\equiv 0 \pmod N$. 
\end{proposition}

\begin{proof}
Observe that $A[\cN]=\cN^{-1}/\oh_K$. The canonical isogeny $\varphi_{\cN} : A\lra A/A[\cN]$ is defined over $H$ and given over $\C$ by the quotient isogeny $\C/\oh_K \lra \C/\cN^{-1}$. Note in particular that $A/A[\cN]$ has CM by $\oh_K$. By definition, $t$ is a generator of the cyclic subgroup $A[\cN]$, hence there exist integers $c$ and $d$ with $\gcd(c,d,N)=1$ such that $t=(c\tau+d)/N + \langle 1, \tau \rangle$. Note that only the classes of $c$ and $d$ modulo $N$ matter in this expression for $t$. If $c\equiv 0 \pmod N$, then $d$ is coprime to $N$, $A[\cN]=\langle 1/N + \langle 1,\tau \rangle\rangle$, and $\varphi_{\cN}$ is the quotient isogeny $\C/\langle 1, \tau \rangle \lra \C/\langle 1/N, \tau \rangle$. But $\C/\langle 1/N, \tau \rangle\simeq \C/\langle 1, N\tau \rangle$ has CM by $\oh_N$ \cite[Lemma 7.5]{cox}, contradicting the fact that $A/A[\cN]$ has CM by $\oh_K$. Thus, $c\not\equiv 0 \pmod N$. 
%Note that since $c/\gcd(c,N)$ is coprime to $N$, we can always multiply by its inverse modulo $N$ in order to get the generator $(\gcd(c, N)\tau+d')/N + \langle 1, \tau \rangle$ with $d'=d(c/\gcd(c,N))^{-1}$ in $\Z/N\Z$.
\end{proof}

\begin{remark}\label{rem:t}
The fact that $A/A[\cN]$ has CM by $\oh_K$ places restrictions on the possible choices of the integer $d$ in Proposition \ref{prop:t}. For instance, in the case where $N$ is prime, the generator $t$ of $A[\cN]$ must, up to multiplication by an element of $(\Z/N\Z)^\times$, be of the form $t=(\tau+d)/N + \langle 1, \tau \rangle$ for one of the two choices of $d\in \Z/N\Z$ in Remark \ref{rem:qsplit} such that $\C/\langle 1, (\tau+d)/N\rangle$ has CM by $\oh_K$. Indeed, in this case $A/A[\cN]=\C/\langle 1, t\rangle$ must have CM by $\oh_K$ and Remark \ref{rem:qsplit} applies since $N$ splits in $K$ by the Heegner hypothesis. 
\end{remark}

%\begin{proof}
%Observe that $A[\cN]=\cN^{-1}/\oh_K$. The canonical isogeny $\varphi_{\cN} : A\lra A/A[\cN]$ is defined over $H$ and given over $\C$ by the quotient isogeny $\C/\oh_K \lra \C/\cN^{-1}$. Note in particular that $A/A[\cN]$ has CM by $\oh_K$. By definition, $t$ is a generator of the cyclic subgroup $A[\cN]$, hence there exist integers $c$ and $d$ with $\gcd(c,d,N)=1$ such that $t=(c\tau+d)/N + \langle 1, \tau \rangle$. Note that only the classes of $c$ and $d$ modulo $N$ matter in this expression for $t$. If $c=0$, then $d$ is coprime to $N$, $A[\cN]=\langle 1/N + \langle 1,\tau \rangle\rangle$, and $\varphi_{\cN}$ is the quotient isogeny $\C/\langle 1, \tau \rangle \lra \C/\langle 1/N, \tau \rangle$. But $\C/\langle 1/N, \tau \rangle\simeq \C/\langle 1, N\tau \rangle$ has CM by $\oh_N$ \cite[Lemma 7.5]{cox}, contradicting the fact that $A/A[\cN]$ has CM by $\oh_K$. In conclusion, $c\neq 0$ modulo $N$ and is thus invertible modulo $N$ under Assumption \ref{Nprime}. By multiplying $t$ by the inverse of $c$ modulo $N$, we may thus assume that $t$ is of the form $(\tau+d)/N + \langle 1, \tau \rangle$ for some $d\in \Z/N\Z$. Following Section \ref{s:qisog}, we then have $A/A[\cN]=\C/\langle 1, t\rangle$ with CM by $\oh_K$. The result follows from Remark \ref{rem:qsplit} which applies since $N$ splits in $K$ by the Heegner hypothesis. 
%\end{proof}

Fix a choice of $\Gamma_1(N)$-level structure $t\in A[\cN]$. By Proposition \ref{prop:t}, $t=(c\tau+d)/N + \langle 1, \tau \rangle$ for some $c, d\in \Z$ with $N \nmid c$ and $\gcd(c,d,N)=1$. Let $a, b, k \in \Z$ be such that $ad-bc-kN=1$ (possible by the gcd condition). Then the matrix 
$\gamma=\gamma_t:=\left(\begin{smallmatrix} 
a & b \\ c & d
\end{smallmatrix}\right)\in \mathrm{M}_2(\Z)$ reduces modulo $N$ into $\SL_2(\Z/N\Z)$. By modifying the entries of $\gamma$ modulo $N$ if necessary, we may and will assume that $\gamma=\left(\begin{smallmatrix} 
a & b \\ c & d
\end{smallmatrix}\right) \in \SL_2(\Z)$. This does not affect $t=(c\tau+d)/N + \langle 1,\tau \rangle$. Note that if $c=1$, then we may for instance take
$\gamma=\left(\begin{smallmatrix} 
1 & d-1 \\ 1 & d
\end{smallmatrix}\right)\in \SL_2(\Z)$. Multiplication by $c\tau+d$ yields an isomorphism 
\[
\varphi_{\cN} : (\C/\langle 1, \gamma(\tau)\rangle, 1/N +\langle 1, \gamma(\tau)\rangle) \overset{\sim}{\lra} (\C/\langle 1, \tau \rangle, (c\tau+d)/N + \langle 1, \tau \rangle)=(A, t)
\]
of elliptic curves with $\Gamma_1(N)$-level structures. It follows that the point $(A, t)\in Y_1(N)(\C)=\Gamma_1(N)\setminus \cH$ is represented by $\Gamma_1(N) \gamma(\tau)=\Gamma_1(N)\frac{a\tau+b}{c\tau+d}$. 

\begin{definition}\label{def:phit}
Let $q$ be an odd prime not dividing $c$, and let $\beta\in \PP^1(\F_q)$. Given the above notations, define $\tau_{q,\beta}^t\in \cH$ to be $q\gamma(\tau)$ if $\beta = \infty$ and $(\gamma(\tau)+\beta)/q$ if $\beta\neq \infty$. Let $\Lambda_{q,\beta}^t:=\langle 1, \tau_{q,\beta}^t \rangle$ and define the isogeny
\[
\varphi_{q,\beta}^t : 
\begin{cases}
\C/\langle 1, \gamma(\tau) \rangle \overset{\quot}{\lra}  \C/\langle 1/q, \gamma(\tau) \rangle \overset{[q]}{\lra} \C/\Lambda_{q,\beta}^t, & \beta=\infty \\
\C/\langle 1, \gamma(\tau) \rangle =  \C/\langle 1, \gamma(\tau)+\beta \rangle \overset{\quot}{\lra} \C/\Lambda_{q,\beta}^t, & \beta\neq\infty.
\end{cases}
\]
\end{definition}

%Note that $c\in \F_q^\times$ since $q\nmid N$.
The composed isogeny 
\begin{equation}\label{def:psit}
\psi_{q,\beta}^t:=\varphi_{q,\beta}^t\circ [(c\tau+d)^{-1}]  : \C/\oh_K \overset{[(c\tau+d)^{-1}]}{\lra} \C/\langle 1, \gamma(\tau) \rangle \overset{\varphi_{q,\beta}^t}{\lra} \C/\Lambda_{q,\beta}^t
\end{equation}
has kernel of size $q$. Hence $\psi_{q,\beta}^t$ must be isomorphic (in the sense of Section \ref{s:isog}) to $\varphi_{q, \beta'}$ of Section \ref{s:qisog} for some $\beta'\in \PP^1(\F_q)$.
Indeed, we have 
\[
\ker(\psi_{q,\beta}^t) =
\begin{cases}
\langle (\tau+c^{-1}d)/q+\langle 1,\tau \rangle  \rangle, & \beta=\infty \\
\langle 1/q+\langle 1,\tau \rangle  \rangle, & a+c\beta\equiv 0 \pmod q \\
\langle (\tau+(a+c\beta)^{-1}(b+d\beta))/q+\langle 1,\tau \rangle  \rangle, & \beta\neq \infty, a+c\beta \not\equiv  0\pmod q,
\end{cases}
\] 
and thus
\[
(\psi_{q,\beta}^t, \C/\Lambda_{q,\beta}^t) =
\begin{cases}
(\varphi_{q,c^{-1}d}, A_{q,c^{-1}d}), & \beta=\infty \\
(\varphi_{q,\infty}, A_{q,\infty}), & a+c\beta\equiv 0 \pmod q \\
(\varphi_{q, (a+c\beta)^{-1}(b+d\beta)}, A_{q, (a+c\beta)^{-1}(b+d\beta)}), & \beta\neq \infty, a+c\beta\not\equiv 0 \pmod q,
\end{cases}
\] 
as elements of $\Isog^{\cN}(A)$. In particular, there is an equality of subsets of $\Isog^{\cN}(A)$
\begin{equation}\label{eq:subsets}
\{ (\varphi_{q,\beta}, A_{q,\beta}) \mid \beta\in \PP^1(\F_q) \} = \{ (\psi_{q,\beta}^t, \C/\Lambda_{q,\beta}^t \rangle) \mid \beta\in \PP^1(\F_q) \},
\end{equation}
and Proposition \ref{prop:trans} can be restated as:

\begin{proposition}\label{prop:transt}
With the above notations, let $q$ be an odd prime which is inert in $K$ and coprime to $cN$. The action of the Galois group $\Gal(H_q/H)$ on the subset $\{ (\psi_{q,\beta}^t, \C/\Lambda_{q,\beta}^t) \mid \beta\in \PP^1(\F_q) \}$ of $\Isog^{\cN}(A)$ is simply transitive. 
\end{proposition}

\subsection{Explicit $pq$-isogenies}

Retain the notations of the previous subsection. In particular, $t=(c\tau+d)/N + \langle 1, \tau \rangle\in A[\cN]$ with $c,d\in \Z$, $\gcd(c,d,N)=1$, $c\not\equiv 0 \pmod N$, and $\gamma=\gamma_t=\left(\begin{smallmatrix} 
a & b \\ c & d
\end{smallmatrix}\right)\in \SL_2(\Z)$.

Let $q$ be an odd prime which is coprime to $cd_K N$. 
Let $p$ be an auxiliary distinct odd prime which is also coprime to $cd_K N$. Consider the lattices $\Lambda^t_{p,q,\beta}:=\langle 1, \tau_{p,q,\beta}^t\rangle$ with $\beta\in \PP^1(\F_q)$ of index $pq$ in $\oh_K$ where $\tau_{p,q,\beta}^t:=p\tau_{q,\beta}^t$. Consider the isogenies $\psi_{p,q,\beta}^t : \C/\oh_K\lra \C/\Lambda_{p,q,\beta}^t$ obtained by composing $\psi_{q,\beta}^t$ defined in \eqref{def:psit} with the map
\[
\C/\langle 1, \tau_{q,\beta}^t\rangle \overset{\quot}{\lra} \C/\langle 1/p, \tau_{q,\beta}^t\rangle \overset{[p]}{\lra} \C/\langle 1, \tau_{p,q,\beta}^t\rangle.
\]
The isogenies $\psi_{p,q,\beta}^t$ have degree $d_{\psi_{p,q,\beta}^t}=pq$ with kernel spanned by $\ker(\psi_{q,\beta}^t)$ together with the point $(\tau+c^{-1}d)/p + \langle 1, \tau\rangle$. 
%Note that $\Lambda_{p,q,\infty}=\oh_{pq}$ is the order of conductor $pq$. 

\begin{proposition}\label{prop:pqinert}
If $q$ is an odd prime which is inert in $K$ and coprime to $cN$, and $p$ is a distinct auxiliary prime not dividing $cd_K N \vert c\tau+d\vert^2$, then $\Lambda_{p,q,\beta}^t$ is a proper fractional $\oh_{pq}$-ideal for all $\beta\in \PP^1(\F_q)$.
\end{proposition}

\begin{proof}
Begin by observing that 
\begin{equation}\label{eq:gammatau}
\gamma(\tau):=\frac{a\tau + b}{c\tau +d}=\frac{\tau}{\vert c\tau+d\vert^2} + \frac{ac\vert\tau\vert^2-bcd_K+bd}{\vert c\tau+d\vert^2}. 
\end{equation}
Consequently, we have $\Q(\tau_{p,q,\beta}^t)=\Q(\tau_{q,\beta}^t)=\Q(\gamma(\tau))=\Q(\tau)=K$. By \eqref{eq:subsets}, $\C/\Lambda_{q,\beta}^t$ is equal to $A_{q,\beta', \C}=\C/\Lambda_{q,\beta'}$ for some $\beta'\in \PP^1(\F_q)$. Since $q$ is inert, it follows by Proposition \ref{prop:inert} that $\Lambda_{q,\beta}^t=\Lambda_{q,\beta'}$ is a proper fractional $\oh_q$-ideal. 

Let us first suppose that $\beta\neq \infty$.
By \cite[Lemma 7.5]{cox}, $\tau_{q,\beta}^t$ must satisfy a quadratic equation of the form $q^2\vert c\tau+d\vert^2 X^2 +AX+B$, with coefficients $A, B \in \Z$ such that $\gcd(A,B, q^2\vert c\tau+d\vert^2)=1$. But then $\tau_{p,q,\beta}^t$ satisfies the equation $q^2\vert c\tau+d\vert^2 X^2 +ApX+Bp^2$. The coefficients have $\gcd$ equal to $1$ since $p$ and $q$ are distinct and $p\nmid \vert c\tau+d\vert^2$. By \cite[Lemma 7.5]{cox}, $\Lambda_{p,q,\beta}^t$ is then a proper fractional $\langle 1, q^2\vert c\tau+d\vert^2\tau_{p,q,\beta}^t\rangle$-ideal. The result follows by observing, using \eqref{eq:gammatau}, that 
\[
q^2\vert c\tau+d\vert^2\tau_{p,q,\beta}^t=pq\vert c\tau+d\vert^2(\gamma(\tau)+\beta)=pq(\tau+ac\vert\tau\vert^2-bcd_K+bd+\beta\vert c\tau+d\vert^2),
\]
hence $\langle 1, q^2\vert c\tau+d\vert^2\tau_{p,q,\beta}^t\rangle=\langle 1, pq\tau\rangle=\oh_{pq}$.

If $\beta=\infty$, then, by \cite[Lemma 7.5]{cox}, $\tau_{q,\infty}^t$ must satisfy a quadratic equation of the form $\vert c\tau+d\vert^2 X^2 +AX+B$, with coefficients $A, B \in \Z$ such that $\gcd(A,B, \vert c\tau+d\vert^2)=1$. But then $\tau_{p,q,\infty}^t$ satisfies the equation $\vert c\tau+d\vert^2 X^2 +ApX+Bp^2$. The coefficients have $\gcd$ equal to $1$ since $p\nmid \vert c\tau+d\vert^2$. By \cite[Lemma 7.5]{cox}, $\Lambda_{p,q,\infty}^t$ is then a proper fractional $\langle 1, \vert c\tau+d\vert^2\tau_{p,q,\infty}^t\rangle$-ideal. The result follows by observing that 
\[
\vert c\tau+d\vert^2\tau_{p,q,\infty}^t=pq\vert c\tau+d\vert^2\gamma(\tau)=pq\tau+pq(ac\vert\tau\vert^2-bcd_K+bd),
\]
hence $\langle 1, \vert c\tau+d\vert^2\tau_{p,q,\infty}^t\rangle=\langle 1, pq\tau\rangle=\oh_{pq}$.
%The statement is clear for $\beta=\infty$, so we assume $\beta\in \F_q$. The element $\tau_{q,\beta}\in \cH$ satisfies the quadratic equation $q^2\tau_{q,\beta}^2+q(d_K-2\beta)\tau_{q,\beta}+(\beta^2-d_K\beta+d_K(d_K+1)/4)=0$ and the coefficients have $\gcd$ equal to $1$ since $q$ is inert (see the proof of Proposition \ref{prop:inert}). It follows that $\tau_{p, q,\beta}\in \cH$ satisfies the quadratic equation $q^2\tau_{p,q,\beta}^2+pq(d_K-2\beta)\tau_{p,q,\beta}+p^2(\beta^2-d_K\beta+d_K(d_K+1)/4)=0$ and the $\gcd$ of the coefficients is $1$ since $p\neq q$. By \cite[Lemma 7.5]{cox}, $\Lambda_{p,q,\beta}$ is then a proper fractional $\langle 1, q^2\tau_{p,q,\beta}\rangle$-ideal. The result follows by observing that $\langle 1, q^2\tau_{p,q,\beta}\rangle=\langle 1, (\tau+\beta)pq\rangle=\oh_{pq}$.
\end{proof}

\begin{proposition}\label{prop:pqfod}
Let $q$ be an odd prime which is inert in $K$ and coprime to $cN$. Let $p$ be a distinct auxiliary odd prime which is coprime to $c d_K N \vert c\tau+d\vert^2$. For all $\beta\in \PP^1(\F_q)$, there exists an elliptic curve $A_{p,q,\beta}^t$ with CM by $\oh_{pq}$ over the ring class field $H_{pq}$ such that $A_{p,q,\beta,\C}^t=\C/\Lambda_{p,q,\beta}^t$ and the isogeny of complex tori $\psi_{p,q,\beta}^t$ descends to an isogeny $\psi_{p,q,\beta}^t : A\lra A^t_{p,q,\beta}$ giving rise to an isomorphism class $(\psi_{p,q,\beta}^t, A_{p,q,\beta}^t)\in \Isog_{pq}^{\cN}(A)$ with field of definition $H_{pq}$.
\end{proposition}

\begin{proof}
By Proposition \ref{prop:pqinert}, the elliptic curves $\C/\Lambda_{p,q,\beta}^t$ have CM by $\oh_{pq}$, and the result is a consequence of the main theorem of complex multiplication \cite[Theorem 11.1]{cox}.
\end{proof}

\begin{proposition}\label{transpq}
Let $q$ be an odd prime which is inert in $K$ and coprime to $cN$. Let $p$ be a distinct auxiliary odd prime which is coprime to $c d_K N \vert c\tau+d\vert^2$.
The Galois group $\Gal(H_{pq}/H_p)$ acts simply transitively on the subset $\{ (\psi_{p,q,\beta}^t, A_{p,q,\beta}^t) \: \vert \: \beta\in \PP^1(\F_q) \} \subset \Isog^{\cN}(A)$. The action is determined by the action of $\Gal(H_q/H)$ on $\{ (\psi_{q,\beta}^t, \C/\Lambda_{q,\beta}^t) \: \vert \: \beta\in \PP^1(\F_q) \}$ upon restricting automorphisms from $H_{pq}$ to $H_q$.
\end{proposition}

\begin{proof}
Let $\beta\in \PP^1(\F_q)$. 
The isogeny $\psi_{p,q,\beta}^t$ has kernel of size $pq$. The $p$-part of this kernel is generated by $(\tau+c^{-1}d)/p + \langle 1, \tau \rangle$, and is therefore independent of $\beta$. It corresponds to the isomorphism class of an isogeny from $A$ to the elliptic curve $A_p:=A_{p, c^{-1}d}$ defined over $H_p$ of Section \ref{s:qisog} (if $p$ is inert in $K$, then $A_p$ has CM by $\oh_p$, while for $p$ split it can happen that $A_p$ has CM by $\oh_K$ as explained in Remark \ref{rem:qsplit}. In any case, the isomorphism class is defined over $H_p$). In particular, $\Gal(H_{pq}/H_p)$ fixes the $p$-part of the isogeny $\psi_{p,q,\beta}^t$, and as a result its action on the set $\{ (\psi_{p,q,\beta}^t, A_{p,q,\beta}^t) \: \vert \: \beta\in \PP^1(\F_q) \}$ is well-defined. Note that the latter set has order $(q+1)/u_K$ (as in the proof of Proposition \ref{prop:trans}, using \eqref{eq:subsets}). By Proposition \ref{prop:pqfod}, we have an inclusion $\{ (\psi_{p,q,\beta}^t, A_{p,q,\beta}^t) \: \vert \: \beta\in \PP^1(\F_q) \}\subset \Isog_{pq}^{\cN}(A)$ (in the notation of Section \ref{s:isog}) and the Galois action is simple.
Since $p$ and $q$ are distinct primes, we have $H_p\cap H_q=H$ and $H_{pq}=H_p\cdot H_q$ \cite[Proposition 1.7]{mythesis}. Thus, the natural restriction map from $H_{pq}$ to $H_q$ induces an isomorphism of Galois group 
\begin{equation}\label{resgal}
\Gal(H_{pq}/H_p)\simeq \Gal(H_q/H).
\end{equation}
By Lemma \ref{lem:qinert}, $\Gal(H_{pq}/H_p)$ thus has order equal to the one of $\{ (\psi_{p,q,\beta}^t, A_{p,q,\beta}^t) \: \vert \: \beta\in \PP^1(\F_q) \}$, and consequently the action is transitive.
\end{proof}

\begin{definition}\label{def:delta}
With the above notations, 
given an odd prime $q$ coprime to $c d_K N$ and inert in $K$, an auxiliary disctinct prime $p$ coprime to $c d_K N \vert c\tau+d\vert^2$, and $\beta\in \PP^1(\F_q)$, the cycles associated to the isomorphism classes $(\psi_{p,q,\beta}^t, A_{p,q,\beta}^t)\in \Isog^{\cN}_{pq}(A)$ are denoted
%isogeny 
% $$\varphi_{p, q, \beta} : A\lra A_{p, q, \beta} $$   
% of degree $pq$  gives rise to the cycles
\[
\tilde{\Delta}^{\HC}_{p, q, \beta}:=\tilde{\Delta}^{\HC}_{\psi_{p, q, \beta}^t} \quad \text{ and } \quad \tilde{\Delta}^{\GHC}_{p, q, \beta}:=\tilde{\Delta}^{\GHC}_{\psi_{p, q, \beta}^t}
\]
in the notation of Definition \ref{rem:intcycles}.
By Propositions \ref{prop:fod} and \ref{prop:pqinert}, they are defined over the field compositum $F_{pq}=K_{\cN}\cdot H_{pq}\subset K^{\ab}\subset \bar \Q$.
\end{definition}
%Note that the isogeny $\varphi_{p, q, \beta}$ has degree $d_{\varphi_{p, q, \beta}}=pq$, and $A_{p, q, \beta}$ has CM by the order of conductor $pq$.
%Let $F_{pq}$ denote the field compositum of $K_{\cN}$ and $H_{pq}$, so that $\Delta^{\HC}_{p, q, \beta}$ is defined over $F_{pq}$. It suffices to show that the subgroup of the Griffiths group generated by the algebraic equivalence classes of the cycles $\Delta^{\HC}_{p, q, \beta}$ for $p, q$ and $\beta$ varying has infinite rank. 

\begin{proposition}\label{prop:galcyc}
Fix a $\Gamma_1(N)$-level structure $t\in A[\cN]$, and
let $p$ and $q$ be distinct odd primes coprime to $c d_K N$ with $q$ inert in $K$ and $p\nmid \vert c\tau+d\vert^2$.
The action of $\Gal(F_{pq}/F_p)$ on the subset $\{ \tilde{\Delta}^{\HC}_{p,q, \beta} \mid \beta\in \PP^1(\F_{pq}) \}$ of $\CH^{r+1}(W_{k, F_{pq}})_0$ is determined by the action of $\Gal(H_{pq}/H_p)$ on the subset $\{ (\psi_{p,q,\beta}^t, A^t_{p,q,\beta}) \: \vert \: \beta\in \PP^1(\F_q) \}$ of $\Isog_{pq}^{\cN}(A)$ under the restriction map $\Gal(F_{pq}/F_p)\lra \Gal(H_{pq}/H_p)$. In particular, the action is transitive.
\end{proposition}

\begin{proof}
Let $\sigma\in \Gal(F_{pq}/F_p)$ and $\beta\in \PP^1(\F_q)$. Let $t_{p,q,\beta}:=\psi^t_{p,q,\beta}(t)\in A_{p,q,\beta}^t[\cN\cap \oh_{pq}]$ and let $\iota_{p,q,\beta}$ denote the inclusion of $(A_{p,q,\beta}^t)^k$ in $W_k$ as the fibre above $P_{p,q,\beta}:=(A_{p,q,\beta}^t, t_{p,q,\beta})\in X_1(N)(F_{pq})$. By Definitions \ref{rem:hc} and \ref{rem:intcycles}, we have $\tilde{\Delta}_{p,q,\beta}^{\HC}:=\tilde{\epsilon}_{W_K} (\iota_{p,q,\beta})_*((\Gamma_{[pq\sqrt{-d_K}]})^r)\in \CH^{r+1}(W_{k,F_{pq}})_0$, and thus 
\[
(\tilde{\Delta}_{p,q,\beta}^{\HC})^\sigma=\tilde{\epsilon}_{W_K}^\sigma (\iota^\sigma_{p,q,\beta})_*((\Gamma_{[pq\sqrt{-d_K}]^\sigma})^r).
\]
We have $\tilde\epsilon_{W_K}^\sigma=\tilde\epsilon_{W_K}$ since $\tilde\epsilon_{W_k}$ is defined over $\Q$, and $[pq\sqrt{-d_K}]^\sigma=[pq\sqrt{-d_K}]\in \End((A_{p,q,\beta}^t)^\sigma)$ since $\sigma$ fixes $H$. The map $\iota^\sigma_{p,q,\beta}$ is the inclusion of $((A^t_{p,q,\beta})^\sigma)^k$ in $W_k$ as the fibre above the point $P_{p,q,\beta}^\sigma=((A^t_{p,q,\beta})^\sigma, (\psi^t_{p,q,\beta})^\sigma(t^\sigma))$ of $
X_1(N)$. Since $\sigma$ fixes $K_{\cN}$, it fixes $(A, t)\in X_1(N)(K_{\cN})$, so that $P_{p,q,\beta}^\sigma=((A^t_{p,q,\beta})^\sigma, (\psi^t_{p,q,\beta})^\sigma(t))$. Thus, the action is determined by the action on $(\psi^t_{p,q,\beta}, A^t_{p,q,\beta})\in \Isog_{pq}^{\cN}(A)$. The last part of the statement then follows from Proposition \ref{transpq}.
\end{proof}

\section{Asymptotics for Abel--Jacobi images of explicit cycles}\label{s:asymptotics}
 
Let $k=2r\geq 2$ and $N\geq 5$ be integers. Fix an imaginary quadratic field $K$ with ring of integers $\cO_K$ and discriminant $-d_K$ coprime to $N$. Assume that $K$ satisfies the Heegner hypothesis with respect to $N$, and let $\cN$ denote a choice of cyclic $N$-ideal of $\oh_K$.
Let $H$ be the Hilbert class field of $K$. Let $A$ be an elliptic curve with CM by $\oh_K$ over $H$, and choose the embedding $H\hookrightarrow \C$ of Section \ref{s:convention} such that $A_\C=\C/\oh_K$. Let $\tau:=(-d_K+\sqrt{-d_K})/2$ denote the standard generator of $\oh_K$ so that $\oh_K=\langle 1,\tau\rangle$, and fix a $\Gamma_1(N)$-level structure $t\in A[\cN]$. Then $t=(c\tau+d)/N + \langle 1,\tau\rangle$ for some $c, d\in \Z$ with $c\not\equiv 0 \pmod N$ and $\gcd(c,d,N)=1$ by Proposition \ref{prop:t}. As in Section \ref{s:levelstruc}, let $a,b\in \Z$ such that $\gamma:=\gamma_t=\left(\begin{smallmatrix} a& b \\ c & d \end{smallmatrix}\right) \in \SL_2(\Z)$ (which might require translating $c$ and $d$ by some multiples of $N$).
%Assume that $(A, t)=(\C/\langle 1,\tau\rangle, 1/N \pmod{\langle 1,\tau\rangle})$ in $Y_1(N)$.
%Fix a $\Gamma_1(N)$-level structure $t\in A[\cN]$ and assume that $t= 1/N \pmod{\oh_K}\in A_\C$. 
%Fix also a choice $\omega_A=2\pi idw$ of non-zero differential form in $H^{1,0}(A_\C)=H^{1,0}(\C/\oh_K)$. Recall that it uniquely determines a generator $\eta_A$ of $H^{0,1}(A_\C)$ by the condition $\langle \omega_A, \eta_A\rangle_A=1$.

Define the indexing set
\[
\mathcal{I}=\mathcal{I}_t:=\{ (p, q) \mid p>q \text{ odd primes coprime to } c d_K \vert c\tau+d\vert^2, q \text{ inert in } K, p,q \equiv 1\pmod N \}.
\]
Note that this set is infinite by Dirichlet's theorem on primes in arithmetic progressions since $d_K$ and $N$ are coprime.
Using Theorem \ref{thm:CAJ}, we are going to produce asymptotic estimates for the Abel--Jacobi images of the Heegner cycles in the collection 
\begin{equation}\label{collec}
\mathcal{C}:=\{ \tilde\Delta^{\HC}_{p,q,\beta}\in \CH^{r+1}(W_{k, F_{pq}})_0 \mid (p,q) \in \mathcal{I}, \beta \in \PP^1(\F_q)  \}
\end{equation}
(see Definition \ref{def:delta}) as $p/q \to \infty$. As a corollary, we will deduce information about the orders of the algebraic equivalence classes of these cycles when $p/q$ is large.
%From now on, we will use brackets to denote the algebraic equivalence class of a cycle. 

\begin{definition}\label{def:improper}
Given $(p,q)\in \mathcal{I}$ and $\beta\in \PP^1(\F_q)$, define 
\[
\gamma_{p,q,\beta}:=
\begin{cases}
pq & \beta=\infty \\
p/q & \beta\neq \infty,
\end{cases}
\qquad \text{ and } \qquad \kappa_{p,q,\beta}:=
\begin{cases}
1 & \beta=\infty \\
q & \beta\neq \infty.
\end{cases}
\]
Writing $\tau_{p,q,\beta}^t=:X^t_{p,q,\beta}+iY^t_{p,q,\beta}$ and
using \eqref{eq:gammatau}, we see that $Y^t_{p,q,\beta}=\gamma_{p,q,\beta}\vert c\tau+d\vert^{-2} \sqrt{d_K}/2$ and
\[
X^t_{p,q,\beta}=
\begin{cases}
\vert c\tau+d\vert^{-2}(ac\vert \tau\vert^2-bcd_K+bd-d_K/2)pq & \beta=\infty \\
(\vert c\tau+d\vert^{-2}(ac\vert \tau\vert^2-bcd_K+bd-d_K/2)+\beta)p/q & \beta\neq\infty.
\end{cases}
\]
Consider the convergent improper integral
\[
I^t_{p,q,\beta}:=\int_{Y^t_{p,q,\beta}}^{\infty}(y^2-(Y^t_{p,q,\beta})^2)^r e^{-2\pi y} dy > 0,
\]
and define $J_{p,q,\beta}:=2^{k+1}\pi^{r+1}i^{r}\vert c\tau+d\vert^k(pq)^r\kappa_{p,q,\beta}^k m_{k,k}^2 e^{2\pi i X^t_{p,q,\beta}}I^t_{p,q,\beta}\in \C^\times$.
\end{definition}

\begin{lemma}\label{lem:est}
Let $d$ be the dimension of $S_{k+2}(\Gamma_1(N))$ and choose a basis $(f_1, \ldots, f_d)$ consisting of normalised cuspidal eigenforms for the action of the Hecke algebra. This choice identifies $S_{k+2}(\Gamma_1(N))^\vee$ with $\C^d$, and we let $L$ denote the lattice in $\C^d$ whose elements are the evaluations (via integration) of the elements of the lattice $L_k'\subset S_{k+2}(\Gamma_1(N))^\vee$ at $\vec{\omega}:=(\omega_{f_1}, \ldots, \omega_{f_d})$. Given $(p,q)\in \mathcal{I}$ and $\beta\in \PP^1(\F_q)$, we view $\AJ_{W_k}((2^k k!)^2(pq)^r\tilde\Delta^{\HC}_{p,q,\beta})$ as an element of $\C^d$ by identifying it with the formula displayed on the right hand side of Theorem \ref{thm:CAJ} (this amounts to choosing a fixed representative of $\AJ_{W_k}((2^k k!)^2(pq)^r\tilde\Delta^{\HC}_{p,q,\beta})$ viewed as an element of $\C^d/L$). Then, as a complex vector valued function of $(p,q)\in \mathcal{I}$, $\AJ_{W_k}((2^k k!)^2(pq)^r\tilde\Delta^{\HC}_{p,q,\beta})$ is coordinate-wise asymptotically equivalent to $\vec{J}_{p,q,\beta}=(J_{p,q,\beta}, \ldots, J_{p,q,\beta})\in \C^d$ as $p/q \to \infty$. 
\end{lemma}

\begin{proof}
By Definition \ref{def:delta}, $\tilde\Delta^{\HC}_{p,q,\beta}$ is the Heegner cycle associated to $(\psi^t_{p,q,\beta}, A^t_{p,q,\beta})\in \Isog_{pq}^{\cN}(A)$. We have $A^t_{p,q,\beta, \C}=\C/\langle 1, \tau^t_{p,q,\beta}\rangle$ with $\tau_{p,q,\beta}^t$ equal to $pq\gamma(\tau)$ or $(\gamma(\tau)+\beta)p/q$ depending on whether $\beta=\infty$ or $\beta\neq \infty$. 
Recall that the isogeny $\psi^t_{p,q,\beta} : \C/\langle 1,\tau\rangle \lra \C/\langle 1,\tau^t_{p,q,\beta}\rangle$ is given by mapping $w \pmod{\langle 1,\tau\rangle} \mapsto  pq(c\tau+d)^{-1}w \pmod{\langle 1,\tau^t_{p,q,\beta}\rangle}$ if $\beta=\infty$ and $w \pmod{\langle 1,\tau\rangle} \mapsto p(c\tau+d)^{-1}w \pmod{\langle 1,\tau^t_{p,q,\beta}\rangle}$ if $\beta\neq \infty$. By the assumption that $p$ and $q$ are both congruent to $1$ modulo $N$, we thus have $\psi^t_{p,q,\beta}(t)=1/N \pmod{\langle 1,\tau^t_{p,q,\beta}\rangle}$.
%By Definition \ref{def:delta}, $\Delta^{\HC}_{p,q,\infty}$ is the Heegner cycle associated to $(\varphi_{p,q,\infty}, A_{p,q,\infty})\in \Isog^{\cN}(A)$. We have $A_{p,q,\infty}=\C/\Lambda_{p,q,\infty}$ with $\Lambda_{p,q,\infty}=\langle 1, \tau_{p,q,\infty}\rangle$ and $\tau_{p,q,\infty}=pq\tau$. 
%Recall that the isogeny $\varphi_{p,q,\infty} : \C/\langle 1,\tau\rangle \lra \C/\langle 1,pq\tau\rangle$ is given by mapping $w \mod \langle 1,\tau\rangle$ to $pqw \mod \langle 1,pq\tau\rangle$. It follows that $\varphi_{p,q,\infty}(t)=pq/N \mod \langle 1,pq\tau\rangle = 1/N \mod \langle 1,pq\tau\rangle$ since both $p$ and $q$ are assumed to be congruent to $1$ modulo $N$. 
%Moreover, we have $\varphi_{p,q,\infty}^*(2\pi i dw)=pq \omega_A$. The generator $\omega_A^\infty:=pq\omega_A$ of $H^{1,0}(A_\C)$ determines the generator $\eta_A^{\infty}:=(pq)^{-1}\eta_A$ via the condition $\langle \omega_A^\infty, \eta_A^\infty\rangle_A=1$. 
%Denote by $d$ the dimension of the complex vector space $S_{k+2}(\Gamma_1(N))$. It admits a basis of normalised cuspidal eigenform $f_1, \ldots, f_d$ with respect to the action of the Hecke algebra. This choice identifies $S_{k+2}(\Gamma_1(N))^\vee$ with $\C^d$, and we let $L$ denote the lattice in $\C^d$ whose elements are the evaluation (via integration) of the elements of $L_k'$ on $\vec{\omega}:=(\omega_{f_1}, \ldots, \omega_{f_d})$. 
%With these identifications in hand, 
Applying Theorem \ref{thm:CAJ} therefore yields the equality
\[
\AJ_{W_k}((2^k k!)^2(pq)^r\tilde\Delta_{p,q,\beta}^{\HC})=(-2)^r (2\pi i)^{r+1} \vert c\tau+d\vert^k (pq)^r \kappa_{p,q,\beta}^k m_{k,k}^2 \int_{i\infty}^{\tau^t_{p,q,\beta}}(z-\tau^t_{p,q,\beta})^r(z-\bar{\tau}^t_{p,q,\beta})^r \vec{f}(z) dz 
 \]
modulo the lattice $L$,
where $\vec{f}=(f_1, \ldots, f_d)$, and $\kappa_{p,q,\beta}$ is defined in Definition \ref{def:improper}. Writing $\tau^t_{p,q,\beta}=X^t_{p,q,\beta}+iY^t_{p,q,\beta}$ as in Definition \ref{def:improper} and making the change of variables $z=X^t_{p,q,\beta} + iy$ gives the equality
\[
\AJ_{W_k}((2^k k!)^2(pq)^r\tilde\Delta_{p,q,\beta}^{\HC})=2^{k+1}\pi^{r+1}i^{r} \vert c\tau+d\vert^k (pq)^r \kappa_{p,q,\beta}^k m_{k,k}^2  \int_{Y^t_{p,q,\beta}}^{\infty}(y^2-(Y^t_{p,q,\beta})^2)^r \vec{f}(X^t_{p,q,\beta}+iy) dy
\]
modulo the lattice $L$.
%Applying Theorem \ref{thm:CAJ} yields the formula
%\[
% \AJ_{W_k}(\Delta_{p,q,\infty}^{\HC})=(-2)^r(2\pi i)^{r+1} \int_{i\infty}^{\tau_{p,q,\infty}}(z-\tau_{p,q,\infty})^r(z-\bar{\tau}_{p,q,\infty})^r \vec{f}(z) dz \pmod {L},
% \]
%where $\vec{f}=(f_1, \ldots, f_d)$. Making the change of variables $z=-pqd_K/2 + iy$ gives
%\[
%\AJ_{W_k}(\Delta_{p,q,\infty}^{\HC})=2^{k+1}\pi^{r+1}i^{r} \int_{pq\sqrt{d_K}/2}^{\infty}(y^2-(pq)^2 d_K/4)^r \vec{f}(-pqd_K/2+iy) dy \pmod {L}
%\]
For all $1\leq j\leq d$, using the Fourier expansion at $i\infty$ together with the fact that $f_j$ is a normalised cuspidal eigenform, there exists a constant $c_j > 0$ such that 
\begin{equation}\label{est:cusp}
\vert f_j(z)-e^{2\pi i z} \vert \leq c_j e^{-4\pi \Im(z)} \qquad \text{ for all } z\in \cH.
\end{equation}
Let $\vec{c}:=(c_1, \ldots, c_d)\in \C^d$.
%Consider the convergent improper integral 
%\[
%I_{p,q,\infty}:=\int_{pq\sqrt{d_K}/2}^{\infty}(y^2-(pq)^2 d_K/4)^r e^{-2\pi y} dy > 0,
%\]
%define
%\[
%J_{p,q,\infty}:=-2^{k+1}\pi^{r+1}i^{r+2}e^{-pqd_K\pi i}I_{p,q,\infty}\in \C,
%\]
%and let $\vec{J}_{p,q,\infty}\in \C^d$ denote the constant vector.
Using \eqref{est:cusp}, Definition \ref{def:improper}, and the fact that $Y^t_{p,q,\beta}=\gamma_{p,q,\beta}\vert c\tau+d\vert^{-2} \sqrt{d_K}/2$, we deduce that
\[
\vert \AJ_{W_k}((2^k k!)^2(pq)^r\tilde\Delta_{p,q,\beta}^{\HC})-\vec{J}_{p,q,\beta} \vert \leq 2^{k+1}\pi^{r+1} \vert c\tau+d\vert^k (pq)^r\kappa_{p,q,\beta}^k m_{k,k}^2 e^{-\gamma_{p,q,\beta} \vert c\tau+d\vert^{-2} \pi\sqrt{d_K}} I^t_{p,q,\beta} \vec{c}.
\]
It follows that 
\[
\left\vert \frac{\AJ_{W_k}((2^k k!)^2(pq)^r\tilde\Delta_{p,q,\beta}^{\HC})-\vec{J}_{p,q,\beta}}{\vec{J}_{p,q,\beta}} \right\vert \leq e^{-\gamma_{p,q,\beta}\vert c\tau+d\vert^{-2}\pi\sqrt{d_K}}\vec{c}.
\]
The result follows by noting that $\gamma_{p,q,\beta}\to \infty$ as $p/q\to \infty$.
%This proves that $\AJ_{W_k}(\Delta_{p,q,\infty}^{\HC})$, seen as a complex vector valued function of $p$ and $q$, is asymptotically equivalent to $\vec{J}_{p,q,\infty}$ as $pq \to \infty$. 
\end{proof}

\begin{proposition}\label{cor:est}
With the above notations, as $p/q$ tends to $\infty$ for $(p,q)\in \mathcal{I}$, the order of $\AJ_{W_k}(\tilde\Delta_{p,q,\beta}^{\HC})$ becomes large (possibly infinite) in the intermediate Jacobian $J^{r+1}(W_{k,\C})$ for all $\beta\in \PP^1(\F_q)$.
\end{proposition}

\begin{proof}
We have
$\vert J_{p,q,\beta} \vert=2^{k+1}\pi^{r+1} \vert c\tau+d\vert^k (pq)^r \kappa^k_{p,q,\beta}m_{k,k}^2 I^t_{p,q,\beta}$ and 
\[
I^t_{p,q,\beta}=\gamma_{p,q,\beta}^{k+1}\int_{\vert c\tau+d\vert^{-2}\sqrt{d_K}/2}^\infty (y^2-\vert c\tau+d\vert^{-4}d_K/4)^r e^{-2\pi\gamma_{p,q,\beta} y} dy.
\] 
Unfolding the power $(y^2-\vert c\tau+d\vert^{-4}d_K/4)^r$ and integrating by parts repeatedly yields the formula
\begin{equation}\label{unfold}
I^t_{p,q,\beta}=\gamma_{p,q,\beta}^{k+1}e^{-\gamma_{p,q,\beta} \vert c\tau+d\vert^{-2} \pi \sqrt{d_K}} \sum_{j=0}^r \sum_{s=0}^{2j} (-1)^{r-j} \binom{r}{j} \frac{(2j)!}{(2j-s)!} \frac{(\vert c\tau+d\vert^{-2}\sqrt{d_K}/2)^{k-s}}{(2\pi\gamma_{p,q,\beta})^{s+1}}.
\end{equation}
It follows that $\vec{J}_{p,q,\beta}$ tends to $0$ in $\C^d$ as $p/q\to \infty$. Let $\epsilon>0$ such that the open polydisc $D_\epsilon(0)^d\subset \C^d$ of radius $\epsilon$ and center $0$ satisfies $D_\epsilon(0)^d\cap L=\{ 0 \}$ (possible by discreteness of $L$). For $p/q$ large enough, $\vec{J}_{p,q,\beta}$ lies in $D_\epsilon(0)^d$ and thus $\vec{J}_{p,q,\beta} \in L$ if and only if $\vec{J}_{p,q,\beta}=0$. However, $\vec{J}_{p,q,\beta}\neq 0$ as is clear from the definition of the integral $I_{p,q,\beta}$. The closer the value $\vec{J}_{p,q,\beta}$ is to $0\in \C^d$, the larger the order (if finite) of $\vec{J}_{p,q,\beta} \pmod L$ becomes in $\C^d/L$. By Lemma \ref{lem:est}, the $j$-th coordinate satisfies $\AJ_{W_k}((2^k k!)^2(pq)^r\tilde\Delta^{\HC}_{p,q,\beta})_j=J_{p,q,\beta}(1+o(1))$ for all $1\leq j\leq d$. In particular, $\AJ_{W_k}((2^k k!)^2(pq)^r\tilde\Delta^{\HC}_{p,q,\beta})$ (viewed in $\C^d$) tends to $0$ without being equal to $0$ as $p/q\to \infty$. We deduce that $(2^k k!)^2(pq)^r \AJ_{W_k}(\tilde\Delta^{\HC}_{p,q,\beta})$ has large (if finite) order in $S_{k+2}(\Gamma_1(N))^\vee/L_k'$. In particular, the same conclusion holds for $\AJ_{W_k}(\tilde\Delta_{p,q,\beta}^{\HC})$. The result follows since $S_{k+2}(\Gamma_1(N))^\vee/L_k'$ is a quotient of $J^{r+1}(W_{k,\C})$.
\end{proof}

\begin{proposition}\label{cor:estbi}
With the above notations, as $p/q$ tends to $\infty$ for $(p,q)\in \mathcal{I}$, the order of $\AJ_{W_k}(\tilde\Delta^{\HC}_{p,q,\beta})-\AJ_{W_k}(\tilde\Delta^{\HC}_{p,q,\infty})$ becomes large (possibly infinite) in the intermediate Jacobian $J^{r+1}(W_{k,\C})$ for all $\infty\neq \beta\in \PP^1(\F_q)$.
%Let $q$ be an odd prime which is congruent to $1$ modulo $N$ and inert in $K$, and let $\infty\neq \beta\in  \PP^1(\F_{q})$. For any constant $C>0$, there exist infinitely many primes $p$ such that $(p,q)\in \mathcal{I}$ and $\AJ_{W_k}(\Delta_{p,q,\infty}-\Delta_{p,q,\beta})$ has  order greater than $C$ in the intermediate Jacobian $J^{r+1}(W_{k,\C})$.
\end{proposition}

\begin{proof}
Define the polynomial 
\[
P(X):=\sum_{j=0}^r \sum_{s=0}^{2j} (-1)^{r-j} \binom{r}{j} \frac{(2j)!}{(2j-s)!} \frac{(\vert c\tau+d\vert^{-2}\sqrt{d_K}/2)^{k-s}}{(2\pi)^{s+1}}X^{k-s} \in \R[X].
\]
It has degree $<k$. By \eqref{unfold}, for all $\beta\in \PP^1(\F_q)$, we have 
\[
I^t_{p,q,\beta}=e^{-\gamma_{p,q,\beta}\vert c\tau+d\vert^{-2}\pi \sqrt{d_K}}P(\gamma_{p,q,\beta}).
\]
It follows that
\[
\left\vert\frac{J_{p,q,\infty}}{J_{p,q,\beta}}\right\vert = q^{-k}\frac{I^t_{p,q,\infty}}{I^t_{p,q,\beta}}
=q^{-k} \frac{P(pq)}{P(p/q)}e^{-\vert c\tau+d\vert^{-2}\pi\sqrt{d_K}(q^2-1)p/q}. 
%=q^{k+2}\int_{\sqrt{d_K}/2}^\infty e^{-2\pi y(q^2-1)p/q} dy=\frac{q^{k+3}}{2\pi p(q^2-1)}e^{-\pi\sqrt{d_K}(q^2-1)p/q}. 
\]
In particular, the limit as $p/q\to \infty$ is equal to $0$. By Lemma \ref{lem:est}, we deduce that for all $1\leq j\leq d$, the ratio of the $j$-th coordinates $\vert \AJ_{W_k}((2^k k!)^2(pq)^r\tilde\Delta^{\HC}_{p,q,\infty})_j/\AJ_{W_k}((2^k k!)^2(pq)^r\tilde\Delta_{p,q,\beta})_j\vert$ tends to $0$ as $p/q\to \infty$. Thus, 
\begin{align*}
\AJ_{W_k}((2^k k!)^2(pq)^r\tilde\Delta^{\HC}_{p,q,\beta})_j -\AJ_{W_k}((2^k k!)^2(pq)^r\tilde\Delta^{\HC}_{p,q,\infty})_j & =\AJ_{W_k}((2^k k!)^2(pq)^r\tilde\Delta^{\HC}_{p,q,\beta})_j(1+o(1)) \\
& =J_{p,q,\beta}(1+o(1)),
\end{align*}
as $p/q\to \infty$. 
From this point on, the result follows from the same argument as in the proof of Proposition \ref{cor:est}.
\end{proof}

\begin{corollary}\label{cor:inford}
There exists a constant $C_r>0$ such that for all $(p,q)\in \mathcal{I}$ satisfying $p/q>C_r$, the following statements hold: 
\begin{itemize}
\item[(1)]  For all $\beta \in \PP^1(\F_q)$, the algebraic equivalence class of $\tilde\Delta^{\HC}_{p,q,\beta}$ in $\Gr^{r+1}(W_{k,F_{pq}})$ has infinite order.
\item[(2)] For all $\infty\neq \beta \in \PP^1(\F_q)$, the algebraic equivalence class of $\tilde\Delta_{p,q,\beta}^{\HC}-\tilde\Delta_{p,q,\infty}^{\HC}$ in $\Gr^{r+1}(W_{k,F_{pq}})$ has infinite order. 
\end{itemize} 
\end{corollary}

\begin{proof}
Given $\beta \in \PP^1(\F_q)$, we let $\Xi_{p,q, \beta}$ denote either $\tilde\Delta^{\HC}_{p,q,\beta}$ or $\tilde\Delta_{p,q,\beta}^{\HC}-\tilde\Delta_{p,q,\infty}^{\HC}$ (in the latter case we exclude $\beta=\infty$).
Using either Proposition \ref{cor:est} or Proposition \ref{cor:estbi}, we may choose a constant $C_r>0$ such that for all $(p,q)\in \mathcal{I}$ satisfying $p/q > C_r$, the order of $\AJ_{W_k}(\Xi_{p,q, \beta})$ in $J^{r+1}(W_{k,\C})$ is greater that the constant $M_r$ of Definition \ref{def:Mr}, for all $\beta \in \PP^1(\F_q)$. Let $(p,q)\in \mathcal{I}$ with $p/q> C_r$ and let $\beta\in \PP^1(\F_q)$. The cycle $\Xi_{p,q,\beta}$ is defined over $F_{pq}$ by Proposition \ref{prop:fod}, and thus its algebraic equivalence class $[\Xi_{p,q,\beta}]$ belongs to $\Gr^{r+1}(W_{k, F_{pq}})$. Suppose by contradiction that [$\Xi_{p,q,\beta}]$ is torsion in $\Gr^{r+1}(W_{k, F_{pq}})$. Using  the Galois equivariance of the Bloch map (Proposition \ref{prop:bloch}) along with Proposition \ref{lem:grbloch}, we see that $\lambda_{W_k}^\circ([\Xi_{p,q,\beta}])$ belongs to $H^{k+1}_{\et}(W_{k,\bar\Q}, \Q/\Z(r+1))^{\Gal(\bar\Q/F_{pq})}$ and is thus annihilated by $M_r$ by Corollary \ref{coro:finite}. Recall the equality \eqref{eq:commu}: $\lambda^\circ_{W_k} = \comp^{-1} \circ u \circ \AJ_{W_k}$.
We conclude that $u(\AJ_{W_k}([\Xi_{p,q,\beta}]))$ is annihilated by $M_r$. By injectivity of the map $u$, we deduce that $\AJ_{W_k}([\Xi_{p,q,\beta}])$ is annihilated by $M_r$, which contradicts that fact that the order of $\AJ_{W_k}([\Xi_{p,q,\beta}])$ is greater than $M_r$. This proves by contradiction that $[\Xi_{p,q,\beta}]$ has infinite order.
\end{proof}

\section{Infinite rank Griffiths groups}\label{s:gr}

Let $k=2r\geq 2$ be an even integer.
Fix an imaginary quadratic field $K$ with ring of integers $\cO_K$ and discriminant $-d_K$ coprime to $N$. Assume that $K$ satisfies the Heegner hypothesis with respect to $N$, and let $\cN$ denote a choice of cyclic $N$-ideal of $\oh_K$.
Let $H$ be the Hilbert class field of $K$. Let $A$ be an elliptic curve with CM by $\oh_K$ over $H$ and choose the embedding $H\hookrightarrow \C$ of Section \ref{s:convention} such that $A_\C=\C/\oh_K$. Let $\tau:=(-d_K+\sqrt{-d_K})/2$ denote the standard generator of $\oh_K$ so that $\oh_K=\langle 1,\tau\rangle$, and fix a $\Gamma_1(N)$-level structure $t\in A[\cN]$. Then $t=(c\tau+d)/N + \langle 1,\tau\rangle$ for some $c, d\in \Z$ with $c\not\equiv 0 \pmod N$ and $\gcd(c,d,N)=1$ by Proposition \ref{prop:t}. As in Section \ref{s:levelstruc}, let $a,b\in \Z$ such that $\gamma:=\gamma_t=\left(\begin{smallmatrix} a& b \\ c & d \end{smallmatrix}\right) \in \SL_2(\Z)$ (which might require translating $c$ and $d$ by some multiples of $N$).
%Fix a $\Gamma_1(N)$-level structure $t\in A[\cN]$. 
%Let $\tau:=(-d_K+\sqrt{-d_K})/2$ denote the standard generator of $\oh_K$ as in the previous section, so that $\oh_K=\langle 1,\tau\rangle$.
%Assume that $(A, t)=(\C/\langle 1,\tau\rangle, 1/N \pmod{\langle 1,\tau\rangle})$ in $Y_1(N)$.
%Let $H$ be the Hilbert class field of $K$. Let $A$ be an elliptic curve with CM by $\oh_K$ over $H$. Assume that the complex embedding $H\hookrightarrow \C$ fixed in Section \ref{s:convention} is such that $A_\C=\C/\oh_K$. Fix a $\Gamma_1(N)$-level structure $t\in A[\cN]$ and assume that $t=1/N \pmod{\oh_K}$.

%Recall from Section \ref{s:intro_gr} that the Griffiths group is the group of null-homologous cycles modulo algebraic equivalence. 
Of interest is the group $\Gr^{r+1}(W_{k, \bar \Q})$ of algebraic equivalence classes of cycles of codimension $r+1$ defined over $\bar \Q$. More precisely, we will focus on the subgroup 
 $G_{\mathcal{C}}^{\HC}$ of $\Gr^{r+1}(W_{k, K^{\ab}})$ generated by the algebraic equivalence classes of the Heegner cycles in the collection $\mathcal{C}$ defined in \eqref{collec}. 
 
\begin{theorem}\label{thm:gr}
With the above notations and assumptions, we have $\dim_{\Q} G^{\HC}_{\mathcal{C}} \otimes_{\Z} \Q=\infty$. 
\end{theorem}

%\begin{remark}
%The method of proof of Theorem \ref{thm:gr} follows closely that of the proof of \cite[Theorem 2]{bdlp}, which itself is an adaptation of the original work and ideas of Schoen \cite{schoen}. Prompted by the referee's suggestion, we will give a self-contained proof which does not assume familiarity with these prior works. As noted already in the introduction, the proof of Theorem \ref{thm:gr} complements the work of Besser \cite{besser95} on complex multiplication cycles on Kuga--Sato varieties over indefinite quaternionic Shimura curves. 
%\end{remark}

\begin{proof}
It suffices to prove that $\dim_{\Q} G^{\HC}_{\mathcal{C}} \otimes_{\Z} \Q \geq \ell-1$ for an arbitrary fixed prime $\ell > 6Nd_K$. Pick an odd prime $q$ which is coprime to $c d_K \vert c\tau+d\vert^2$, congruent to $1$ modulo $N$, inert in $K$, and such that $(q+1)/u_K \equiv 0 \pmod \ell$. Recall that $u_K=\vert\oh_K^\times\vert/2\in \{ 1,2,3 \}$ and in particular it is coprime to $\ell$.  The last condition is thus equivalent to $q\equiv -1 \pmod \ell$. The last three conditions on $q$ are equivalent to a single congruence condition modulo $Nd_K\ell$ by the Chinese Remainder Theorem since $N,d_K$, and $\ell$ are pairwise coprime. In particular, there are infinitely many possible choices for $q$ by Dirichlet's theorem on arithmetic progressions. A single choice of a prime $q$ will suffice for this proof. 

Recall from Lemma \ref{lem:qinert} that the extension $H_q/H$ is cyclic of degree $(q+1)/u_K$. Since $\ell$ divides this degree by assumption on $q$, $\Gal(H_q/H)$ admits a unique cyclic subgroup $G_\ell$ of order $\ell$. Let $\sigma_\ell$ denote a choice of generator of $G_\ell$. By Proposition \ref{prop:transt}, we have 
\begin{equation}\label{betal}
(\psi^t_{q,\infty},\C/\Lambda^t_{q,\infty})^{\sigma_\ell}=(\psi^t_{q,\beta_\ell},\C/\Lambda^t_{q,\beta_\ell})
\end{equation}
in $\Isog_q^{\cN}(A)$
for some $\infty\neq \beta_\ell \in \PP^1(\F_q)$. 

Pick a prime $p$ such that $(p,q)\in \mathcal{I}$ and $p/q$ is greater than the constant $C_r$ of Corollary \ref{cor:inford}. This choice guarantees that the algebraic equivalence classes
\begin{equation}\label{inforderyay}
[\tilde\Delta^{\HC}_{p,q,\infty}], [\tilde\Delta^{\HC}_{p,q,\beta_\ell}], \text{ and } [\tilde\Delta^{\HC}_{p,q,\infty}]-[\tilde\Delta^{\HC}_{p,q,\beta_\ell}] \text{ have infinite order in } \Gr^{r+1}(W_{k, F_{pq}}).
\end{equation}

Restriction of automorphisms induces isomorphisms
\[
\Gal(F_{pq}/K_{\cN})\overset{\sim}{\underset{\text{\cite[(1.39)]{mythesis}}}{\lra}} \Gal(H_{pq}/H)
\qquad \text{ and } \qquad \Gal(H_{pq}/H_p)\overset{\sim}{\underset{\eqref{resgal}}{\lra}} \Gal(H_q/H).
\]
Let $\tilde G_{\ell}\subset \Gal(F_{pq}/K_{\cN})$ be the preimage of $G_\ell$ under the above maps. It is a cyclic subgroup of $\Gal(F_{pq}/F_p)$ of order $\ell$. Denote by $\tilde \sigma_\ell$ the preimage of $\sigma_\ell$, which is a generator of $\tilde{G}_\ell$.  

Define a homomorphism of $\Q$-vector spaces 
\[
\Psi : \Q[\tilde{G}_\ell] \lra \Gr^{r+1}(W_{k,F_{pq}})\otimes_{\Z} \Q, \qquad \tilde{\sigma}_{\ell} \mapsto [(\tilde\Delta^{\HC}_{p,q,\infty})^{\tilde{\sigma}_{\ell}}].
\]
The kernel of $\Psi$ is stable under multiplication by $\Q[\tilde{G}_\ell]$ and is thus an ideal of $\Q[\tilde{G}_\ell]$. Let $\zeta_\ell\in \bar \Q$ be a choice of primitive $\ell$-th root of unity. There is a ring isomorphism
%\footnote{We take the opportunity to point out that there is a typo in the definition of this map in \cite[page 413]{bdlp}.} 
\[
\Q[\tilde{G}_\ell] \overset{\sim}{\lra} \Q \times \Q(\zeta_\ell), \qquad \sum^{\ell-1}_{i=0} \lambda_i \tilde{\sigma}_{\ell}^i \mapsto \left(\sum^{\ell-1}_{i=0} \lambda_i, \sum^{\ell-1}_{i=0} \lambda_i \zeta_\ell^i\right). 
\]
The only proper ideals of the ring $\Q \times \Q(\zeta_\ell)$ are $\{ 0 \}\times \Q(\zeta_\ell)$ and $\Q\times \{ 0 \}$, corresponding respectively to the augmentation ideal and the ideal $\Q N$ in the group ring $\Q[\tilde{G}_\ell]$, where $N=\sum_{i=0}^{\ell-1} \tilde{\sigma}_\ell^i$ is the norm element.

By Propositions \ref{prop:galcyc} and \ref{transpq}, the action of $\tilde{\sigma}_\ell$ on $\tilde\Delta^{\HC}_{p,q,\infty}$ is determined by the action of $\sigma_\ell$ on $(\psi^t_{q,\infty}, \C/\Lambda^t_{q,\infty})$. It follows from \eqref{betal} that
$
(\tilde\Delta^{\HC}_{p,q,\infty})^{\tilde{\sigma}_\ell} = \tilde\Delta^{\HC}_{p,q,\beta_\ell}.
$
By \eqref{inforderyay}, both $\Psi(1)$ and $\Psi(\tilde{\sigma}_\ell-1)$ are not equal to $0$. Thus, $\ker(\Psi)$ is neither all of $\Q[\tilde{G}_\ell]$ nor the augmentation ideal. In particular, it must either be trivial or equal to $\Q N$, which implies that $\dim_\Q \Q[\tilde{G}_\ell]/\ker(\Psi) \geq \ell-1$. 
\end{proof}

 \begin{corollary}\label{coro}
With the above notations and assumptions, let $0\leq k'=2r' \leq k$ be another even integer and let $X_{k,k'}:=W_{k,H}\times_H A^{k'}$. The subgroup of $\Gr^{r+r'+1}(X_{k,k', K^{\ab}})$ generated by the algebraic equivalence classes of the variants of generalised Heegner cycles $\Delta_{k,k', \psi^t_{p,q,\beta}}$ indexed by $(p,q)\in \mathcal{I}$ has infinite rank.
 \end{corollary}
 
 \begin{proof}
This follows by combining Theorem \ref{thm:gr} and Proposition \ref{prop:relation}. 
\end{proof}

\begin{remark}
The method of Section \ref{s:caj} can be used to give a formula for the complex Abel--Jacobi images of the variants of generalised Heegner cycles $\Delta_{k,k', \varphi}$ with $(\varphi,A')\in \Isog^{\cN}(A)$, as these are images of generalised Heegner cycles under certain correspondences by \cite[Proposition 4.1.1]{bdp3}. Such a formula can then directly be used to prove Theorem \ref{coro}. We have opted not to do so, as it is enough to know the images of Heegner cycles under the complex Abel--Jacobi map in order to deduce Corollary \ref{coro}.
\end{remark}

%\begin{remark}
%As pointed out in the introduction, Theorems \ref{thm:gr} and \ref{coro} recover results of Burungale \cite{burungale20} using a completely different method, which does not rely on the $p$-adic Gross--Zagier formula for generalised Heegner cycles of Bertolini, Darmon, and Prasanna \cite{bdp1}. The complex geometric method presented here is in particular more direct.
%\end{remark}

\subsection*{Acknowledgements}
It is a pleasure to thank Henri Darmon and Ari Shnidman for relevant suggestions and interesting conversations. The author thanks the anonymous referee for constructive criticism and for suggesting to make the proof of Theorem \ref{thm:gr} self-contained. During the preparation of this article, the author was partially funded by an Emily Erskine Endowment Fund Postdoctoral Research Fellowship at the Hebrew University of Jerusalem.

\phantomsection
%\addcontentsline{toc}{section}{References}
\bibliographystyle{plain}
\bibliography{HeegnerCycles_arXivV3.bib}

\end{document}